%% file: main.tex
\documentclass{article}
\usepackage[utf8]{inputenc}
\usepackage[UKenglish]{babel}
\usepackage{amsmath}	% Mathe: Für align-Umgebung
\usepackage{amsfonts}   % Mathe für mathbb
\usepackage{enumitem}   % Custom enumerate
\usepackage{amsthm}
\usepackage{tikz-cd}    % Kommutierende Diagramme
\usepackage{MnSymbol}   % mapsto nach links
\usepackage{ulem}
\usepackage{mathtools}
\usepackage[UKenglish]{babel}
\usepackage{faktor}
\usepackage{tikzsymbols}
\usepackage{pifont}
\usepackage{bbding}
\usepackage[colorlinks]{hyperref}
\usepackage{tikz}
\usetikzlibrary {angles,quotes}
\usepackage{comment}

\title{Infinitary primitive positive definability over the real numbers with convex relations}
\author{Sebastian Meyer\footnote{Funding statement: Funded by the European Union (ERC, POCOCOP, 101071674). Views and opinions expressed are however those of the author(s) only and do not necessarily reflect those of the European Union or the European Research Council Executive Agency. Neither the European Union nor the granting authority can be held responsible for them. }}
\date{\today}

\input{Abkuerzungen}

\begin{document}

\maketitle
\input{abstract}
\tableofcontents
\input{00Intoduction}
\input{01Definitions}
\input{02Convex-Classification}

\input{03Constructions}
\input{04Reverse-Construction}
\input{05Non-Implications}

%hier weiteres einfügen
\input{06Logical-Corollaries}
\input{07Thanks}

\bibliography{Quellen} 
\bibliographystyle{alpha} %alphabetic  ieeetr

\end{document}

%% file: Abkuerzungen.tex
\newtheorem{lemma}{Lemma}[section]
\newtheorem{thm}[lemma]{Theorem}
\newtheorem{cor}[lemma]{Corollary}

\newtheorem*{lemma*}{Lemma}
\newtheorem*{lem*}{Lemma}
\newtheorem*{prop*}{Proposition}
\newtheorem*{rem*}{Remark}
\newtheorem*{thm*}{Theorem}

\theoremstyle{definition}
\newtheorem{definition}[lemma]{Definition}
\newtheorem{notation}[lemma]{Notation}
\newtheorem{rem}[lemma]{Remark}
\newtheorem{example}[lemma]{Example}

\newenvironment{proofof}[1]{\begin{proof}[Proof of #1]}{\end{proof}}

\newcommand{\IR}{\mathbb{R}}
\newcommand{\IQ}{\mathbb{Q}}
\newcommand{\IN}{\mathbb{N}}

\newcommand{\conv}{\operatorname{conv}}

\renewcommand{\epsilon}{\varepsilon}
\renewcommand{\subset}{\subseteq}
\renewcommand{\supset}{\supseteq}
\newcommand{\mitte}{\ \middle| \ }

\newcommand{\structA}{{\underline{A}}}
\newcommand{\setstructA}{A} %Grundmenge der Struktur \structA

\newcommand{\de}[1]{\textbf{#1}}

\newcommand{\closure}[1]{\overline{#1}} % topologischer Abschluss
\newcommand{\inside}[1]{{#1}^\circ}

\newcommand{\var}[1]{y_{#1}} % durchlaufende Variablen für Formeln
\newcommand{\inner}{u} %innerer Punkt, essentiell innerer Punkt
\newcommand{\vv}{v}
\newcommand{\xx}{w}

\newcommand{\closedball}{\overline{B}}
\newcommand{\ball}{B}

\newcommand{\rect}{R_{\dot{\square}}}
\newcommand{\stripe}{R_{\dot{=}}}

\newcommand{\filt}{\mathcal{F}}

\newcommand{\dent}{semident}

%% file: abstract.tex
\begin{abstract}
    On a finite structure, the polymorphism invariant relations are exactly the primitively positively definable relations. On infinite structures, %such as the real numbers, 
    these two sets of relations are different in general. Infinitary primitively positively definable relations are a natural intermediate concept which extends primitive positive definability by infinite conjunctions.
    
    We consider for every convex set $S\subset \IR^n$ the structure 
    of the real numbers $\IR$ with addition, scalar multiplication, constants, and additionally the relation $S$. We prove that depending on $S$, the set of all relations with an infinitary primitive positive definition in this structure equals one out of six possible sets. 
    This dependency gives a natural partition of the convex sets into six nonempty classes. 
    We also give an elementary geometric description of the classes  
    and a description in terms of linear maps.

    The classification also implies that 
    there is no locally closed clone between the clone of affine combinations and the clone of convex combinations.
\end{abstract}

%% file: 00Intoduction.tex
\section{Introduction}
This article considers infinitary primitive positive definability over the real numbers with some convex relations. It can be read from a geometrical viewpoints or in the context of logic and definability. 
We combine these two perspectives by considering a structure on the real numbers 
on which %the closure operator of 
infinitary primitive positive definability has a geometrical interpretation. In our setting, a family of subsets of $\IR^n$ is closed under infinitary primitive positive definitions if and only if it is closed under affine maps, preimages of affine maps, and infinite intersections.

%Infinitary primitive positive definability can be understood as a closure operator. In our setting, a geometric definition is that a family of subsets of $\IR^n$ is closed under infinitary primitive positive definability if it is closed under affine maps, preimages of affine maps, and infinite intersections. 
The geometric characterisation provides a convenient way to analyse certain conditions on closed families of relations. 
For example, if the family contains a closed square, then it also has to contain a closed circle, because the square can be rotated by any angle and the intersection of all these rotations is a circle. Conversely, if the family contains a closed circle, then it also contains a closed square. Indeed, a closed circle can be projected to a compact interval, and the preimage of this projection is a stripe. By intersecting two of these stripes, we obtain a square. Thus, it follows that a closed square and a closed circle have the same closure with respect to infinitary primitive positive definitions. We therefore call them infinitary  primitively positively interdefinable. The precise definition of this equivalence relation is given in Section \ref{SectionDefinitions}. 

%We determine the equivalence classes of convex subsets of $\IR^n$ with respect to infinitary primitive positive interdefinability. 

Infinitary  primitive positive interdefinability gives a natural distinction of convex sets. 
Surprisingly, there are only finitely many convex sets which are pairwise non-interdefinable.
%the number of its equivalence classes is finite. 
More precisely, we show that every convex subset of $\IR^n$ is infinitary primitively positively interdefinable with exactly one of the following six sets:
\begin{enumerate}
    \item A point $\{0\}\subset \IR$,
    \item a nontrivial compact interval $[0,1]\subset \IR$,
    \item a bounded open interval $(0,1)\subset \IR$,
    \item a pointed rectangle $(-1,1) \times (0,1) \cup \{(0,0)\} \subset \IR^2$, %\{(x,y) \in \IR^2 \mid (x\in (-1,1)\land y \in (0,1)) \lor (x,y)=(0,0)\}$,
    \item a pointed stripe $\IR \times (0,1) \cup \{(0,0)\}\subset \IR^2$, or %$\{(x,y) \in \IR^2 \mid y \in (0,1) \lor (x,y)=(0,0)\}$, or        
    \item a halve line $[0,\infty)\subset \IR$.
\end{enumerate}

For the precise results, see Section \ref{SectionResultsShort} where we also include a geometric description of each equivalence class and an independent description in terms of linear maps.

The next three sections, Sections \ref{SectionProofConstructing} to \ref{SectionNonImplications}, contain the main part of the proof of the classification while the last section, Section \ref{SectionCorollaries}, presents the final collection of the proof and additional results and applications. 
The proof is phrased in the language of definability, but the results are also presented in geometrical language.

%% file: 01Definitions.tex
\section{General definitions} \label{SectionDefinitions}
This section introduces common notation used to describe the results.

\subsection{Notions from linear algebra and analytic geometry}
%There are some notions from linear algebra and analytic geometry which will be fixed here.
We recall the following basic notions from linear algebra and analytic geometry for clarity.
\begin{definition}
    A set $S\subset \IR^n$ is \de{convex} if for all $x,y\in S$ and all $\lambda \in [0,1]$ also the point $\lambda x + (1-\lambda)y$ is in $S$.
    %We call a set $S\subset \IR^n$ \de{\stronglyConvex}, if for all $a\notin S$, there is a vector $v$ such that $\{s-a,v\}>0$ for all $s\in S$.
\end{definition}
\begin{notation}
    Consider a point $v\in \IR^n$ and $r>0$.
    The expression $v_i$ refers to the $i$-th component of $v$, so $v=(v_1,\dots ,v_n)^T$.
    The scalar product $\IR^n \times \IR^n \to \IR$ will be denoted by $\langle \bullet, \bullet\rangle$.
%\end{notation}
%\begin{notation}
    %Consider a point $v\in \IR^n$ and $r>0$. 
    The set of points with distance less than $r$ from $v$ 
    %open ball with center $x$ and radius $r$ 
    is denoted by $\ball_r(v)$. 
    %The closed ball is denoted $\closure{\ball}_r(x)$.
    The set of points with distance at most $r$ from $v$ is denoted by $\closure{\ball}_r(v)$. 
    The topological closure of a set $S\subset \IR^n$ is denoted by $\closure{S}$. Thus, $\closure{\ball_r(v)}=\closedball_r(v)$.
    The powerset of any set $X$ is denoted by $\powerset(X)$.
\end{notation}
\begin{definition}
    An \de{affine map} is a map $f\colon \IR^n\to \IR^m$ which can be described as $v\mapsto Av+w$ for a fixed vector $w\in \IR^m$ and an $m\times n$-matrix $A$.
    
    For sets $S\subset \IR^m$, $T\subset \IR^n$ and an affine map $f\colon \IR^n\to \IR^m$, we define the image of $T$ under $f$ as $f(T)\coloneqq \{f(t)\in \IR^m \mid t\in T\}$ and the \de{inverse map} $f^{-1}$ applied to $S$ as $f^{-1}(S)\coloneqq \{t\in \IR^n \mid f(t)\in S\}$.
\end{definition}

\subsection{Definability and polymorphisms}
%There are some notions from universal algebra and model theory which will be fixed here.
The following notions are standard from universal algebra and model theory. 
%We repeat them for a better understanding.

\begin{definition}
    A \de{signature} $\sigma$ is a set where each element is a function symbol with an arity in $\IN\cup\{0\}$ or a relation symbol with an arity in $\IN \setminus \{0\}$.

    A \de{structure} $\structA$ over a signature $\sigma$ consists of a set $\setstructA$ and a function $f_\structA\colon \setstructA^n \to \setstructA$ for each function symbol $f$ in $\sigma$ of arity $n>0$ 
    and a relation $R_\structA \subset \setstructA^n$ for each relation symbol $R$ in $\sigma$ of arity $n$. For each function symbol of arity 0, $\structA$ contains a constant in $A$ which we view as the image of a map $A^0\to A$.
    A detailed definition is given in \cite[Section 1]{Hodges}.
    %Often, $\sigma$ will be clear from the context or we write $(\setstructA,\sigma)$ for the structure. 
\end{definition}
\begin{definition}
    We denote the real numbers with $\IR$.
    We will often use $\IR$ as base set for the structure and identify relation symbols $R$ with the corresponding relations $R_\IR\subset \IR^n$ on the real numbers and function symbols $f$ with the functions on $\IR$ as long as they are clear from the context. 
    To specify $\sigma$, we will write $(\IR; s\mid s\in \sigma)$ instead of $\sigma$-structure on base set $\IR$. 
    
    On the real numbers, we usually consider variants of the structure 
    $(\IR; +,\cdot c \mid c \in \IR, c \mid c \in \IR, S)$ consisting of the vector space structure $(\IR; +,\cdot c \mid c \in \IR)$, a constant symbol $c \in \IR$ on every real number and a relation symbol $S$ which is interpreted as a convex subset of $\IR^n$.
\end{definition}
\begin{definition}
    Let $\alpha$ be an ordinal number.
    A \de{polymorphism} of arity $\alpha$ of a $\sigma$-structure $\structA$ is a map $f\colon \setstructA^\alpha \to \setstructA$ with the following properties:
    \begin{itemize}
    \item
        For every relation symbol $R$ of arity $m$ in $\sigma$ and every $\alpha$-indexed sequence $(x^i)_{i\in \alpha}$ of elements in $R_\setstructA \subset \setstructA^m$, the image 
        $$
            f((\setstructA^i)_{i\in \alpha}) \coloneqq (f((x^i_1)_{i\in \alpha} ) ,\dots, f((x^i_m)_{i\in \alpha}) ) \in \setstructA^m
        $$
        is again in $R_\setstructA$. 
    \item 
        For every function symbol $g$ of arity $m$ in $\sigma$ and every $\alpha$-indexed sequence $(x^i)_{i\in \alpha}$ of elements in $\setstructA^m$, the equality
        $$
            g(f((x^i)_{i\in \alpha})) \coloneqq g(f((x^i_1)_{i\in \alpha} ) ,\dots, f((x^i_m)_{i\in \alpha}) ) 
            = f((g(x^i))_{i\in \alpha})
        $$
        holds. 
    \end{itemize}
    A polymorphism of a $\sigma$-structure $\structA$ has \de{finite arity}, if its arity $\alpha$ is a natural number. In this case, the numbering $(1,2,\dots,\alpha)$ instead of $(0,1,2,\dots,\alpha-1)$ will be used. Note that in the literature, the word polymorphism often refers to finite arity polymorphisms.
\end{definition}
\begin{definition} \label{DefinitionIPPDefinable}
    Let $\structA$ be a $\sigma$-structure. A subset $S$ of $\setstructA^n$ is called
    \begin{itemize}
        \item \de{pp-definable} or primitively positively definable in $\structA$
        if there is a primitive positive formula (that is a first order formula which only uses $\land, \exists$ and relations from $\sigma$ and which avoids $\lor, \forall, \lnot$) defining $S$.
        \item \de{polymorphism invariant} if for every polymorphism $p\colon \setstructA^m\to \setstructA$ of finite arity and every $s_1,\dots ,s_m\in S$ also $(p(s_{1,1},\dots ,s_{m,1}),\dots , p(s_{1,n},\dots ,s_{m,n}))$ is in $S$ where $s_{i,j}$ denotes the $j$-th component of $s_i$.
        \item \de{infinitary pp-definable} %the Bodirsky, Seite 176
        or infinitary primitively positively definable
        if there is an infinitary primitive positive formula defining $S$ \cite[Section 6.1.2]{theBodirsky}. That is, if the formula defining $S$ is inside the smallest set containing 
        \begin{itemize}
            \item true and false,
            \item atomic formulas from $\sigma$,
            \item all formulas $\phi$ from this set where the variables are substituted,
            \item all finite and infinite conjunctions of formulas from this set whenever the conjunction has a finite number of free variables,
            \item the formula $\exists x: \phi$ whenever $\phi$ is in this set.
        \end{itemize}
    \end{itemize}
\end{definition}

Those definitions are connected by the following implications:
\begin{lemma} \label{LemmaDefinablePolyStable}
    Let $\structA$ be a $\sigma$-structure and $S\subset \setstructA^n$.
    \begin{enumerate}
        \item If $S$ is pp-definable then $S$ is infinitary pp-definable.
        \item If $S$ is infinitary pp-definable then $S$ is polymorphism invariant and also invariant under all infinite arity polymorphisms.
    \end{enumerate}
\end{lemma}
The first implication is trivial. The implication ``If $S$ is pp definable then $S$ is invariant under finite arity polymorphisms" is shown in \cite[Proposition 6.1.3]{theBodirsky}. The second implication of the lemma is stronger but can be proven analogously. If $\setstructA$ is finite, this three properties coincide \cite[Theorem 2]{Geiger}.

\begin{definition}
    Let $\structA$ be a $\sigma$-structure. Then, two sets $S\subset \setstructA^n, T\subset \setstructA^m$ are \de{infinitary pp-interdefinable}, if $S$ is infinitary pp-definable from $\sigma \cup \{T \}$ and $T$ is infinitary pp-definable from $\sigma \cup \{ S \}$.
\end{definition}
\begin{definition}
    A \de{clone} on a set $\setstructA$ is a collection of maps from powers $\setstructA^n$ to $\setstructA$ which contains all projections 
    and is closed under composition in the sense that if $f\colon \setstructA^m\to \setstructA$ and $g_1,\dots,g_m\colon \setstructA^n\to \setstructA$ are in the clone, then also the map 
    $$
        \setstructA^n \to \setstructA, (x_1,\dots,x_n) \mapsto f(g_1(x_1,\dots,x_n),\dots,g_m(x_1,\dots,x_n))
    $$
    is contained in the clone.
    
    A clone $C$ is called \de{locally closed}, if whenever a map $f\colon \setstructA^n \to \setstructA$ has the property that for all finite $S \subset \setstructA^n$ there is $f'\colon \setstructA^n \to \setstructA$ in $C$ such that $f'$ equals $f$ on $S$ then $f$ is already in $C$.
\end{definition}
\begin{thm}[{\cite[page 21]{Szendrei}} or {\cite[Corollary 6.1.6]{theBodirsky}}] \label{TheoremLocallyClosedClone}
    The set of all finite arity polymorphisms of a structure defines a locally closed clone. Every locally closed clone can be obtained that way.
\end{thm}
This clone will be called \de{polymorphism clone}.

%There is also a geometric interpretation of infinitary pp definability when identifying a subset of $\IR^n$ with a relation.

%\begin{thm} \label{TheoremIPPGeometrisch1} \label{TheoremIPPGeometrisch2}
%    Let $A=\bigsqcup_{n\in \IN} \powerset(\IR^n)$ the set of all subsets of any $\IR^n$. Let $\tau \subset A$ be any subset. 
%    
%    Then, $\tau$ is closed under infinite pp-definability, if and only if
%    \begin{enumerate}
%        \item $\tau$ is closed under arbitrary intersections,
%        \item $\tau$ is closed under rearrangeing coordinates,
%        \item $\tau$ is closed under coordinate parallel projections and
%        \item $\tau$ is closed under coordinate parallel projection preimages
%    \end{enumerate}
%    
%    If $(+,\cdot c \mid c \in \IR, c\mid c \in \IR)$ are in $\tau$, then $\tau$ is closed under infinite pp-definability, if and only if is closed under affine maps (including projections), inverse affine maps and (possibly infinite) intersections.
%\end{thm}

\subsection{Newly introduced definitions}
To describe the equivalence classes of infinitary pp-interdefinability geometrically, we will need the following additional definitions. They are original to this text. 
% as the author is unaware of any reference which has introduced this concept before.

\begin{definition}
    Let $S\subset \IR^n$ be a convex set and $a\notin S$. We call $a$ a \de{\dent{}}, if there is $s\in S, t\in \closure{S}$ and $\lambda\in(0,1)$ such that $a=\lambda s + (1-\lambda) t$.
\end{definition}
\begin{definition}
    We say that a convex set $S\subset \IR^n$ has a \de{ray intersection}, if there is an affine map $\IR \to \IR^n, x \mapsto v+x w$ such that $v+xw \in S$ for all $x>0$ and $v+xw\notin S$ for all $x<0$.
\end{definition}

\begin{example}
    The set $\{(x,y) \in \IR^2 \mid (x\in (-1,1)\land y \in (0,1)) \lor (x,y)=(0,0)\}$ 
    $$
    \begin{tikzpicture}
        \fill [gray!20!white] (-1,1) -- (1,1) -- (1,0) -- (-1,0) -- (-1,1);
        \draw [very thin, dashed] (-1,1) -- (1,1);
        \draw [very thin, dashed] (-1,0) -- (1,0);
        \draw [very thin, dashed] (-1,0) -- (-1,1);
        \draw [very thin, dashed] (1,0) -- (1,1);
        \draw [fill] (0,0) circle [radius=2pt];
        \draw [-Latex] (0.5,-0.5) --(0.5,0);
    \end{tikzpicture}
    $$
    has a \dent{} for example at $(0.5,0)$. It has no ray intersection.
\end{example}

%% file: 02Convex-Classification.tex
\section{Results overview}
\label{SectionResultsShort}
Our main result is the classification of convex sets up to infinitary pp-interdefinability.

%More Precisely, we have the following classifications:
\begin{thm} \label{TheoremDescriptionOfEqClasses}
    A convex set $S\subset \IR^n$ is infinitary pp-interdefinable over $(\IR;+,\cdot c \mid c \in \IR, c\mid c \in \IR)$
    \begin{enumerate}
        \item with $\{0\}\in \IR$ 
        $$
        \begin{tikzpicture}
            \draw [fill] (0,0) circle [radius=2pt];
            \draw [very thin, dashed] (-5,0) -- (5,0);
        \end{tikzpicture}
        $$
        if and only if it is empty or an affine subspace.
        \item with $\{x \in \IR \mid x \in [0,1]\}$ 
        $$
        \begin{tikzpicture}
            \draw [very thin, dashed] (-5,0) -- (-1,0);
            \draw [very thin] (-1,0) -- (1,0);
            \draw [very thin, dashed] (1,0) -- (5,0);
            \draw [fill] (-1,0) circle [radius=2pt];
            \draw [fill] (1,0) circle [radius=2pt];
        \end{tikzpicture}
        $$
        if and only if it is not an affine subspace but the sum $\{b+v\mid b \in B, v \in V\}$ of a compact convex set $B\subset \IR^n$ and a vector subspace $V\subset \IR^n$.
        \item with $\{x \in \IR \mid x \in (0,1)\}$ 
        $$
        \begin{tikzpicture}
            \draw [very thin, dashed] (-5,0) -- (-1,0);
            \draw [very thin] (-1,0) -- (1,0);
            \draw [very thin, dashed] (1,0) -- (5,0);
            \draw [fill=white] (-1,0) circle [radius=2pt]; 
            \draw [fill=white] (1,0) circle [radius=2pt];
        \end{tikzpicture}
        $$
        if and only if it is not closed, there is no \dent{} $a\notin S$, and $S$ is the sum $\{b+v\mid b \in B, v \in V\}$ of a bounded, convex set $B\subset \IR^n$ and a vector subspace $V\subset \IR^n$.
        %, it is not closed cannot be chosen to be compact and there is no \dent{} $a\notin S$.
        \item with $\{(x,y) \in \IR^2 \mid (x\in (-1,1)\land y \in (0,1)) \lor (x,y)=(0,0)\}$ 
        $$
        \begin{tikzpicture}
            \fill [gray!20!white] (-1,1) -- (1,1) -- (1,0) -- (-1,0) -- (-1,1);
            \draw [very thin, dashed] (-1,1) -- (1,1);
            \draw [very thin, dashed] (-1,0) -- (1,0);
            \draw [very thin, dashed] (-1,0) -- (-1,1);
            \draw [very thin, dashed] (1,0) -- (1,1);
            \draw [fill] (0,0) circle [radius=2pt];
        \end{tikzpicture}
        $$
        if and only if it is the sum $\{b+v\mid b \in B, v \in V\}$ of a bounded, convex set $B\subset \IR^n$ and a vector subspace $V\subset \IR^n$ and there is a \dent{} $a\notin S$.
        \item with $\{(x,y) \in \IR^2 \mid y \in (0,1) \lor (x,y)=(0,0)\}$ 
        $$
        \begin{tikzpicture}
            \fill [gray!20!white] (-5,1) -- (5,1) -- (5,0) -- (-5,0) -- (-5,1);
            \draw [very thin, dashed] (-5,1) -- (5,1);
            \draw [very thin, dashed] (-5,0) -- (5,0);
            \draw [fill] (0,0) circle [radius=2pt];
        \end{tikzpicture}
        $$
        if and only if it is not the sum of a vector space and a bounded set, and does not contain a ray intersection. 
        \item with $\{x \in \IR \mid x\ge 0\}$ 
        $$
        \begin{tikzpicture}
            \draw [very thin] (0,0) -- (5,0);
            \draw [very thin, dashed] (0,0) -- (-5,0);
            \draw [fill] (0,0) circle [radius=2pt];
        \end{tikzpicture}
        $$
        if and only if it contains a ray intersection.
    \end{enumerate}
    Every convex set is in exactly one of the above classes. Moreover, the sets which are infinitary pp-definable over the structure $(\IR;+,\cdot c \mid c \in \IR, c\mid c \in \IR, S)$ are exactly the sets in the classes with a lower or equal number than $S$. 
    No non-convex set is infinitary pp-definable from convex sets.
\end{thm}

Before giving some examples in Section \ref{SectionExamples}, we want to give another main theorem.
%We give some examples of convex sets and their equivalence class in Section \ref{SectionExamples}.
There is also a geometric interpretation of infinitary pp-definability, which allows the following description of the classification:
\begin{thm} \label{TheoremMainGeometrisch}
%\label{TheoremIPPGeometrisch2}
    Consider the structure $(\IR;+,\cdot c \mid c \in \IR, c\mid c \in \IR)$ and a set $\tau \subset \bigsqcup_{n\in \IN} \powerset(\IR^n)$ of some subsets of real vector spaces. Assume that $\tau$ contains only convex sets. Then, the following are equivalent:
    \begin{enumerate}
        \item \label{TheoremMainGeometrisch1}
        The set $\tau$ is closed under infinitary pp-definitions over the structure $(\IR;+,\cdot c \mid c \in \IR, c\mid c \in \IR)$.
        \item \label{TheoremMainGeometrisch2}
        The set $\tau$ has the following properties:
        \begin{itemize}
            \item It contains the empty set and $\IR^1$.
            \item For each $S\subset \IR^n\in \tau$ and each affine map $f\colon \IR^n \to \IR^m$, the set $f(S)$ is in $\tau$.
            \item For each $T\subset \IR^m\in \tau$ and each affine map $f\colon \IR^n \to \IR^m$, the set $f^{-1}(T)$ is in $\tau$.
            \item For each $n$ and each subset $I \subset \tau\cap {\powerset(\IR^n)}$, the intersection $\bigcap_{S\in I} S\subset \IR^n$ is in $\tau$.
        \end{itemize}
        \item The set $\tau$ is one of the following 6 sets:
        \begin{itemize}
            \item All affine sets,
            \item All sums $\{b+v\mid b\in B, v\in V\}$ of a vector space $V$ and a compact convex set $B$,
            \item All sums $\{b+v\mid b\in B, v\in V\}$ of a vector space $V$ and a bounded convex set $B$ where the sum has no \dent{},
            \item All sums $\{b+v\mid b\in B, v\in V\}$ of a vector space $V$ and a bounded convex set $B$,
            \item All convex sets without a ray intersection or
            \item All convex sets.
        \end{itemize}
        \item 
        There is a subset of the set $\{ f\colon \IR^\IN \to \IR\mid f((c)_{n\in \IN})=c \}$ of linear maps such that the set $\tau$ is the family of all convex sets that are preserved by all of those maps.
    \end{enumerate}
    The equivalence of \ref{TheoremMainGeometrisch1} and \ref{TheoremMainGeometrisch2} also holds for non convex sets.
\end{thm}

Since infinitary pp-definability implies polymorphism invariance, we easily get the following corollary for polymorphisms.

\begin{thm} \label{TheoremClassificationUpPolymorphisms}
    Let $C$ be a locally closed clone on $\IR$ such that $C$ is contained in the finite arity polymorphisms of $(\IR;+,\cdot c \mid c\in \IR, c\mid c \in \IR)$ and $C$ contains the finite arity polymorphisms of the structure of all convex relations. Then $C$ is one of the following two clones:
    \begin{enumerate}
        \item The polymorphism clone of $(\IR;+,\cdot c \mid c\in \IR, c\mid c \in \IR)$ which equals the polymorphism clone of all affine relations.
        In this clone, all elements $f\colon \IR^n \to \IR$ can be described as
        $$
            (a_1,a_2,\dots ,a_n) \mapsto (\lambda_1a_1+\lambda_2a_2+\dots +\lambda_n a_n)
        $$
        where $\lambda_1,\lambda_2,\dots ,\lambda_n\in \IR$ can be any constants such that $\lambda_1+\lambda_2+\dots +\lambda_n = 1$.
        \item The polymorphism clone of the structure containing all convex relations.
        In this clone, all elements $f\colon \IR^n \to \IR$ can be described as
        $$
            (a_1,a_2,\dots ,a_n) \mapsto (\lambda_1a_1+\lambda_2a_2+\dots +\lambda_n a_n)
        $$
        where $\lambda_1,\lambda_2,\dots ,\lambda_n\ge 0$ and $\lambda_1+\lambda_2+\dots +\lambda_n = 1$.
    \end{enumerate}
    In particular, if $p$ is a finite arity polymorphism of $(\IR;+,\cdot c \mid c\in \IR, c\mid c \in \IR)$ that preserves any non-affine convex set, then it preserves all non-affine convex sets.
\end{thm}

%One should note that this proof of Theorem \ref{TheoremClassificationUpPolymorphisms} is not the shortest possible proof because five of the equivalence classes we consider collapse. 

\subsection{Examples} \label{SectionExamples}
We want to give some examples to visualise the definitions of the equivalence classes in Theorem \ref{TheoremDescriptionOfEqClasses}.

\begin{example}
    Some convex sets and their equivalence classes with respect to pp-interdefinability.
    \begin{itemize}
        %\item The sets given in Theorem \ref{TheoremDescriptionOfEqClasses} are each contained in their  equivalence class.
        \item The open ray $\{x\in \IR\mid x>0 \}$
        $$
        \begin{tikzpicture}
            \draw [very thin] (0,0) -- (5,0);
            \draw [very thin, dashed] (0,0) -- (-5,0);
            \draw [fill=white] (0,0) circle [radius=2pt];
        \end{tikzpicture}
        $$
        is infinitary pp-interdefinable with the closed ray.
        \item The half-open interval $\{x\in \IR\mid 0\le x< 1 \}$
        $$
        \begin{tikzpicture}
            \draw [very thin, dashed] (-5,0) -- (-1,0);
            \draw [very thin] (-1,0) -- (1,0);
            \draw [very thin, dashed] (1,0) -- (5,0);
            \draw [fill] (-1,0) circle [radius=2pt]; 
            \draw [fill=white] (1,0) circle [radius=2pt];
        \end{tikzpicture}
        $$
        is infinitary pp-interdefinable with the open interval.
        \item The set $\{(x,y)\in \IR^2 \mid (0<y<1) \lor (y=0\land x> 0)\}$ 
        $$
        \begin{tikzpicture}
            \fill [gray!20!white] (-5,1) -- (5,1) -- (5,0) -- (-5,0) -- (-5,1);
            \draw [very thin, dashed] (-5,1) -- (5,1);
            \draw [very thin, dashed] (-5,0) -- (0,0);
            \draw [very thin] (0,0) -- (5,0);
            \draw [fill=white] (0,0) circle [radius=2pt];
        \end{tikzpicture}
        $$
        is infinitary pp-interdefinable with the ray. Its closure $\{(x,y)\in \IR^2\mid 0\le x\le 1\}$ is infinitary pp-interdefinable with the compact interval.
        \item The set $\{(x,y,z,w)\in \IR^4\mid x^2+y^2<1 \lor (x^2+y^2=1 \land xw=yz)\}$ is infinitary pp-interdefinable with the pointed stripe. It is not a finite union of sums of bounded sets and vector spaces.
        \item Every compact convex set which consists of at least two points has infinitely many points and is infinitary pp-interdefinable with a compact interval.
        \item The set $\{(x,y)\in \IR^2 \mid 0\le x\le 1\land -1\le y\le 1 \land (x<0 \lor x^2+y^2<1)$ 
        $$
        \begin{tikzpicture}
            %\fill [gray!20!white] (-1,1) -- (1,1) -- (1,0) -- (-1,0) -- (-1,1);
            \fill [gray!20!white] (0,0) -- (1,0) arc [start angle=0, end angle=90, radius=1] -- (-1,1) -- (-1,0) -- cycle;
            \draw [very thin] (-1,1) -- (0,1);
            \draw [very thin] (-1,0) -- (1,0);
            \draw [very thin] (-1,0) -- (-1,1);
            \draw [very thin, dashed] (1,0) arc [start angle=0, end angle=90, radius=1];
            \draw [fill=white] (0,1) circle [radius=2pt];
            %\draw [fill=white] (1,0) circle [radius=2pt];
        \end{tikzpicture}
        $$
        does not contain $(0,1)$ but every line through this point intersects the set. 
        However, this is not a \dent{} (and there is no other \dent{}), so it is infinitary pp-interdefinable with the open interval.
    \end{itemize}
\end{example}

\subsection{Overview of the proof}
In the following sections, we will prove Theorem \ref{TheoremDescriptionOfEqClasses} and its variants. 
%As we prove a complete classification, we will do similar steps a couple of times. 
In Section \ref{SectionProofConstructing}, we show for each equivalence class how to infinitary pp-define the representative of that equivalence class. 
This is followed by Section \ref{SectionProofReverse} in which we show how to infinitary pp-define all elements of an equivalence class when given the representative of that class. The third part of the proof, Section \ref{SectionNonImplications}, shows that elements from different equivalence classes are not infinitary pp-interdefinable. All of this is combined in Section \ref{SectionCorollaries} to complete the proof. 
%It is in theory also possible to start reading there and backtrack all of the proofs but we recommend the order from front to back.

%% file: 03Constructions.tex
\section{Defining the representatives}
\label{SectionProofConstructing}
This section shows how to define the representatives of the infinitary pp-interdefinability equivalence relation from a given convex relation.

\subsection{Geometric considerations}
We start by considering and classifying closed convex sets. The results will be applicable to the closure of any convex set as:
%We recall the following lemma:
\begin{lemma} \label{LemmaClosureStaysConvex}
    The closure of a convex set %$S\subset \IR^n$ 
    is convex.
\end{lemma}
Let us now focus on another lemma: %We use it to prove another lemma:
\begin{lemma} \label{LemmaMaximalVectorforClosed}
    Let $S\subset \IR^n$ be a closed convex subset. 
    Then, the following sets are equal:
    \begin{enumerate}
        \item Every inclusion-maximal linear subspace $V_1$ of $\IR^n$ such that for all $s \in S, v_1 \in V_1$ also $s+v_1 \in S$.
        \item The set $V_2$ of all vectors $v_2$ such that for all $s\in S$, we have $s+v_2\in S$ and $s-v_2\in S$.
        \item For all $s\in S$ the set $V_s$ of all $v$ such that for all $\lambda \in \IR$, we have $\lambda v+s\in S$.
    \end{enumerate}
\end{lemma}
\begin{definition} \label{DefinitionInnerVspace}
    Let $S$ be a convex set. We define the \de{inner vector space} of $S$ as the vector space from Lemma \ref{LemmaMaximalVectorforClosed} applied to the closed convex set $\closure{S}$. 
\end{definition}
\begin{proofof}{Lemma \ref{LemmaMaximalVectorforClosed}}
    It suffices to pick one $V_1$ and one $V_s$ since the second condition is independent of any choice.
    
    Clearly, $V_1\subset V_2$.
    
    The set $V_2$ is by definition closed under addition and closed under taking negations. Since $S$ is convex, it is closed under taking scalars in the interval $[0;1]$. Together, the addition, the negations and the scalar multiplication in $[0;1]$ imply that $V_2$ is closed under every scalar multiplication. Therefore, $V_2$ is a vector space, so $V_1=V_2$ and $V_2\subset V_s$.
    
    For $V_s\subset V_2$, note that clearly $0\in V_2$ and $0\in V_s$. For $v\ne 0, v\in V_s$, take any $s'\in S$ and look at the affine sequence $s, s+v, s+2v, s+3v, \dots $. Since this sequence diverges, its distance to $s'$ will be greater than $\|v\|$ for all large enough entries. Take those entries $s+mv$ and look at the point on the line segment from $s'$ to $s+mv$ which have distance $\|v\|$ from $s'$. 
    $$
    \begin{tikzpicture}
        \draw (-2,0) -- (-3,1);
        \draw (-2,0) -- (-1,1);
        \draw (-2,0) -- (1,1);
        \draw (-2,0) -- (3,1);
        \draw (-2,0) -- (5,1);
        \draw [->] (-3,1) -- (6,1);
        \draw [dotted] (-2,0) circle (2);
        \draw (-2,0) node [anchor=north]{$s'$};
        \draw (-3,1) node [anchor=south]{$s$};
        \draw (-1,1) node [anchor=south]{$s+v$};
        \draw (1,1) node [anchor=south]{$s+2v$};
        \draw (3,1) node [anchor=south]{$s+3v$};
        \draw (5,1) node [anchor=south]{$s+4v$};
        \draw [fill] (0,0) circle [radius=2pt];
        \draw (0,0) node [anchor=north west]{$s'+v$};
    \end{tikzpicture}
    $$    
    Those points are in $S$ since $S$ is convex and they converge to $s'+v$. Since $S$ is closed, $s'+v\in S$. Taking $-v$ instead of $v$ shows $s'-v\in S$. This construction works with every nonzero $v\in V_s$ and every $s'\in S$, so $V_s\subset V_2$.
\end{proofof}
\begin{lemma} \label{LemmaDichotomieClosed}
    Let $S\subset \IR^n$ be a closed convex subset. Then, exactly one of the following holds:
    \begin{enumerate}
        \item The set $S$ is the sum $\{x+y \mid x\in C, y \in V\}$ of a compact set $C$ and the inner vector space $V$ or
        \item the set $S$ has a ray intersection.
    \end{enumerate}
\end{lemma}
\begin{proof}
    Clearly, it cannot be both, since the intersection of the affine line with the sum of a compact set and a vector space is compact or everything.
    Let $S$ be any closed convex set. %Define $V$ as inner vector space of $V$ and recall Lemma \ref{LemmaMaximalVectorforClosed}. 
    
    Assume that there is a bounded set $B$ such that $S$ is contained in $B+V$. Then take $C$ as the intersection of the closure $\closure{B}$ of $B$ with $S$. Now, $C$ is closed and bounded and thus compact. Moreover, $C\subset S$ and thus $C+V\subset S$. Actually, we get $S=C+V$: Take any $x\in S$.  Then, there is $v\in V$ such that $x+v\in B \subset C$. Therefore $x\in C+V$.
    
    Assume otherwise. Then, we have for each bounded set $B$ an element $x\in S$ such that $(\{x\}+V)\cap B=\emptyset$. 
    
    Take any $s \in S$. Let $B_n$ be the bounded sets $\ball_n(s)=\{x \in \IR^n \mid \|x-s\|<n\}$ and $x_n$ such that $(\{x_n\}+V)\cap B_n = \emptyset$. Moreover, choose $y_n$ to be the point in $\{x_n\}+V$ such that $y_n-s$ is orthogonal to all points in $V$. 
    Now, let $y_n'$ be the projection of $y_n$ onto the sphere in $\IR^n$ with radius 1 and centre $s$. Since the bounded sphere is compact, the sequence $(y'_n)_{n\in \IN}$ has an accumulation point $a$.
    If $a-s$ would be in $V$, then the definition of $y_n$ would imply that $y_n$, $s$ and $a$ are an acute triangle. This cannot be true for all $n$ by the accumulation property.
    
    By Lemma \ref{LemmaMaximalVectorforClosed} and $a-s\notin V$, we get that there is $\lambda  \in \IR$ such that $\lambda (a-s) + s \notin S$. On the other hand, $\lambda(a-s)+s$ is in $S$ for all $\lambda>0$, because $\lambda(a-s)+s$ can be approximated by line segments from $s$ to $x_n$ and $S$ is convex and closed. Therefore there is a maximal negative real number $y$ such that $z(a-s)+s\notin S$ for all $z<y$. Now, the affine map $\IR \to \IR^n, x \mapsto (x+y)(a-s)+s$ is a ray intersection.
\end{proof}

\begin{lemma} \label{LemmaEssentialInnerPoint}
    Let $S\subset \IR^n$ be a nonempty convex set.
    For a number $m\in \IN$, the following are equivalent:
    \begin{enumerate}
        \item 
        There is an affine space $W$, $S\subset W\subset \IR^n$ such that $m$ is the dimension of $W$ and there is no affine space of lower dimension with this property.
        \item  
        The number $m$ is maximal such that there are $m+1$ points in $S$ and their convex hull is $m$ dimensional.
    \end{enumerate}    
    Moreover, for a point $s\in S$, the following are equivalent:
    \begin{enumerate}[label=\Alph*., ref=\Alph*] 
        \item \label{ConditionLemmaEssentialInnerPoint1}
        There is an affine space $W$, $S\subset W\subset \IR^n$, and an open set $U$ containing $s$ such that $U\cap W \subset S$.
        \item \label{ConditionLemmaEssentialInnerPoint2}
        With $m$ from above, there are $m+1$ points $v_1,\dots ,v_{m+1}\in S$ with an $m$ dimensional convex hull such that $s$ can be written as $\sum_{i=1}^{m+1} \lambda_i v_i$ where $\sum_{i=1}^{m+1} \lambda_i = 1$ and for all $i$, $\lambda_i$ is positive real.
    \end{enumerate}
\end{lemma}
\begin{definition} \label{DefinitionEssInnerPoint}
    Let $S\subset \IR^n$ be a nonempty convex set and $m\in \IN, s\in S$ satisfy any of the equivalent conditions from Lemma \ref{LemmaEssentialInnerPoint}. Then, we call $s$ an \de{essentially inner point} of $S$ and $m$ the \de{outer dimension} of $S$.
\end{definition}
\begin{proofof}{Lemma \ref{LemmaEssentialInnerPoint}}
    The conditions on $m$ are just equivalent characterisations of the dimension of the affine space generated by the points in $S$. 
    
    For $s$, \ref{ConditionLemmaEssentialInnerPoint2} implies \ref{ConditionLemmaEssentialInnerPoint1}: Let $W$ be the affine hull of $S$. 
    Note that $W$ is the affine hull of $\{\lambda_i \mid i \in \{1,\dots ,m\}\}$. Let $U'$ be the linear subspace of $\IR^n$ which is the orthogonal complement of $W$, that is $\{u\in \IR^n\mid \forall v,w\in W: u \text{ and } v-w \text{ are orthogonal}\}$. 
    Then, $U\coloneqq U'+\{\sum_{i=1}^m \lambda_i v_i \mid \lambda_i>0, \sum_{i=1}^m \lambda_i=1\}$ is open and $U\cap A$ is contained in the convex hull of the $\lambda_i$ which is inside $S$.
    
    For $s$, \ref{ConditionLemmaEssentialInnerPoint1} implies \ref{ConditionLemmaEssentialInnerPoint2}: Since $W$ contains $S$ is the only inner condition on $W$, we can assume that $W$ is the affine hull of $S$. Now, since $U$ is open, we may choose $\epsilon>0$ such that $U$ contains a closed $\epsilon$-ball around $s$. Therefore, $W\cap U$ contains the $m$-dimensional $\epsilon$-ball around $s$. Take $m$ orthogonal vectors $v_1,\dots ,v_m$ in $W$ of length $\frac{1}{\sqrt{m}}\epsilon$. Then, we get
    $$
        s= \left(\sum_{i=1}^m \frac{1}{m+1} (s+v_i) \right) + \frac{1}{m+1} \left(s+\sum_{i=1}^m -v_i\right)
    $$
    so $s$ is in the hull of the points $\{s+v_i\mid i\in\{1,\dots ,m\}\} \cup \{s-\sum_{i=1}^m v_i\}$.
\end{proofof}
\begin{example}
    Consider the subset %$\{(x,y,z)\in \IR^3 \mid 0\le x\le 1, 0\le y\le 1, z=0\}=$
    $[0,1]\times [0,1]\times \{0\}$ of $\IR^3$.
    %$$
    %\begin{tikzpicture}
    %    \fill [gray!20!white] (2,1) -- (3,1) -- (3,0) -- (2,0) -- cycle;
    %    \draw [black, very thin] (2,1) -- (3,1) -- (3,0) -- (2,0) -- cycle;
        %\draw [dashed, ->] (0,0) -- (0.5,0.5);
        %\draw [dashed, ->] (0,0) -- (0,1);
        %\draw [dashed, ->] (0,0) -- (1,0);
        %\draw [very thin] (0,1) -- (1,1);
        %\draw [very thin] (0,0) -- (1,0);
        %\draw [very thin] (0,0) -- (-1,1);
        %\draw [very thin] (1,0) -- (1,1);
    %\end{tikzpicture}
    %\in \IR^3
    %$$
    The interior of this set is empty, but its essentially inner points are $(0,1)\times (0,1)\times \{0\}$. It has outer dimension 2.
\end{example}
\begin{cor} \label{CorExistInnerPoint}
    Let $S$ be a nonempty convex subset of $\IR^n$. Then, $s$ has an essentially inner point.
\end{cor}
\begin{proof}
    Let $m$ be the outer dimension of $S$. Then, there are $m+1$ points in $S$ defining an $m$-dimensional simplex. The centre of this simplex is an essentially inner point of $S$.
\end{proof}
\begin{cor} \label{CorInnerPointsDense}
    Let $S$ be a nonempty convex subset of $\IR^n$. Then, the closure of the essentially inner points of $S$ is $\closure{S}$.
\end{cor}
\begin{proof}
    Take any $s\in S$. Let $m$ be the outer dimension of $S$. Then, we can find $m$ other points $s_1,\dots ,s_m$ in $S$ such that $s,s_1,\dots ,s_m$ are affinely independent. Now, every point $\lambda s + \sum_{i=1}^m \lambda_i s_i$ with $\lambda>0, \forall i: \lambda_i>0$ and $\lambda + \sum_{i=1}^m \lambda_i=1$ is an essentially inner point. Thus $s$ is in the closure of the essentially inner points. Therefore, the closure of the essentially inner points is a closed set containing $S$. Thus, it is $\closure{S}$.
\end{proof}
\begin{lemma} \label{LemmaInnenGegenRand}
    Let $S$ be a nonempty convex subset of $\IR^n$, $s\in \closure{S}$ and $\inner \in S$ an essentially inner point of $S$. Then, the open line segment from $s$ to $\inner$ is in $S$.
\end{lemma}
\begin{proof}
    We may assume $s\ne \inner$.

    Let $W$ be the affine hull of $S$ and $\epsilon>0$ such that the ball $\ball_\epsilon(\inner)\cap W$ is contained in $U$. Let $v=\lambda \inner + (1-\lambda) s$ for any $\lambda \in (0,1)$ be an arbitrary point on the line from $\inner$ to $s$.
    
    Now, $s \in \closure{S}$, so there is $q\in S$ with $\|q-s\|<\epsilon\cdot \frac{1-\lambda}{\lambda}$. 
    By the inverse intercept theorem, the point $p$ defined as the intersection of the line through $v$ and $q$ and the parallel line to the line through $s$ and $q$ through $\inner$, has distance to $\inner$ of at most $\epsilon$. 
    $$
        \begin{tikzpicture}
            \draw (-1.2,-1) -- (0,-1) node [anchor=north]{$u\in S$} -- (0,2) node [anchor=east]{$s\in \closure{S}$} -- (0.5,2);
            \draw (-1,-1) node [anchor=north]{$p\in S$} -- (0.5,2) node [anchor=west]{$q\in S$};
            \draw (0,1) node [anchor=east]{$v$};
            \draw [dotted] (0,-1) circle (1.2);
        \end{tikzpicture}
    $$
    It follows that $p\in W\cap \ball_\epsilon(\inner)\subset S$. By the convexity, $v=\lambda p + (1-\lambda)q \in S$. Since $\lambda$ was arbitrary, the whole line segment is in $S$.
\end{proof}

\begin{lemma} \label{LemmaStrahlenVerschieben}
    Let $S$ be a nonempty convex subset of $\IR^n$, $s\in \closure{S}$ and $v\in \IR^n$ such that for all $\lambda>0$ also $s+\lambda v \in \closure{S}$. Let $\inner$ be an essentially inner point of $S$. Then $u+\lambda v\in S$ for all $\lambda>0$.
\end{lemma}
\begin{proof}
     If the affine dimension of $S$ is zero, the claim follows trivially. For $\inner = s$, the claim follows from Lemma \ref{LemmaInnenGegenRand} applied to $\inner + 2\lambda v\in \closure{S}$ and $\inner\in S$.
    
    Otherwise, let $W$ be the affine closure of $S$ and $\epsilon>0$ such that the ball $\ball_\epsilon(\inner)\cap W$ is contained in $U$. %and $\epsilon<2\|\inner-s\|$. 
    Assume $\inner\ne s$. 
    It suffices to show $\inner+v\in S$ by replacing $\lambda v$ with $v$.

    $$
    \begin{tikzpicture}
        \draw [->] (0,1) -- (4,1);
        \draw (0,1) -- (0,-0.5);
        \draw [->] (0,0) -- (4,0);
        \draw (0,-0.5) -- (3,1);
        \draw (0,1) node [anchor=east]{$s\in \closure{S}$};
        \draw (0,0) node [anchor=east]{$u\in S$};
        \draw (0,-0.5) node [anchor=east]{$p\in S$};
        \draw (1,0) node [fill=white, anchor=south]{$u+v$};
        \draw (1,1) node [anchor=south]{$s+v$};
        \draw (3,1) node [anchor=south] {$s+\lambda v\in \closure{S}$};
        \draw [dotted] (0,0) circle (0.7);
    \end{tikzpicture}
    $$
    
    Note that $W$ is closed and therefore $s\in W$. Let $p$ be a point on the line from $s$ to $\inner$ such that $\|p-\inner\| < \epsilon\}$ and $\inner$ is between $p$ and $s$. Note that $p\in \ball_\epsilon(\inner)\cap W$ is an essentially inner point. Take $\lambda >1$ to be the quotient $\|p-s\|/\|p-\inner\|$. Note that $\inner+v$ is an element of the line segment from $p$ to $s+\lambda v$. Since $p$ is inner and $s+\lambda v\in \closure{S}$, Lemma \ref{LemmaInnenGegenRand} implies that $\inner+v\in S$.
\end{proof}
\begin{cor} \label{CorollaryMoveToInner}
    Let $S$ be a nonempty convex subset of $\IR^n$, such that the closure of $S$ has a ray intersection. Then also $S$ has this property.
\end{cor}
\begin{proof}
    Recall that the closure of a convex set is convex, Lemma \ref{LemmaClosureStaysConvex}.
    Let $v,w$ such that $x\mapsto v+xw$ is a ray intersection for $\closure{S}$. By Corollary \ref{CorExistInnerPoint}, there is an essentially inner point $\inner\in S$. 
    By Lemma \ref{LemmaStrahlenVerschieben}, $\inner+\lambda w \in S$ for all $\lambda>0$. On the other side, there is $\lambda<0$ such that $\inner+\lambda w \notin \overline{S}$, because $\overline{S}$ is a closed convex set and $w\notin V_s$ with the notation of Lemma \ref{LemmaMaximalVectorforClosed} as $w\notin V_v$ is not in the inner vector space.
    Let $v'$ be $\inner+\mu w$ where $\mu \in \IR$ is chosen such that $\inner+\lambda w \in S$ for $\lambda>\mu$ and $\inner+\lambda w \notin S$ for $\lambda<\mu$. Then $x\mapsto v'+xw$ is a ray intersection for $S$.
\end{proof}
\begin{thm} \label{TheoremClosureImpliesRay} \label{TheoremClosureIffRay}
    Let $S$ be a convex subset of $\IR^n$. Then, $S$ has a ray intersection if and only if there is an affine space $A$ such that the closure $\closure{A\cap S}$ is not the sum of a compact set and a vector space. 
\end{thm}
\begin{proof}
    If $x\mapsto v+xw$ is a ray intersection, then the affine set $\{v+\lambda w\mid \lambda \in \IR\}$ gives a closure which looks line a closed ray and is not the sum of a bounded set and a vector space.
    
    For the other direction, fix $A$ such that $\closure{A\cap S}$ is not the sum of a compact set and a vector space. Note that $S\cap A$ and $\closure{S\cap A}$ are convex by Lemma \ref{LemmaClosureStaysConvex}.
    
    By applying Lemma \ref{LemmaDichotomieClosed}, $\closure{S\cap A}$ has a ray intersection. By Corollary \ref{CorollaryMoveToInner} there are also $v,w\in \IR^n$ such that $v+xw\in {S\cap A}$ for all $x>0$ and $v+xw\notin {S\cap A}$ for all $x<0$. Since $A$ is affine, we get $v+xw\in {A}$ for all $x\in \IR$, so $v+xw\notin {S}$ for all $x<0$. Thus, $S$ has a ray intersection.
\end{proof}
\begin{rem}
    The set $A$ can often be chosen to be $\IR^n$. This works whenever $\closure{S}$ is not the sum of a bounded set and a vector space.

    However, it might happen that $\closure{S}$ is the sum of a vector space and a bounded set: The set $\{(x,y)\in \IR^2 \mid (0<y<1) \lor (y=0\land x> 0)\}$ 
    $$
        \begin{tikzpicture}
            \fill [gray!20!white] (-5,1) -- (5,1) -- (5,0) -- (-5,0) -- (-5,1);
            \draw [very thin, dashed] (-5,1) -- (5,1);
            \draw [very thin, dashed] (-5,0) -- (0,0);
            \draw [very thin] (0,0) -- (5,0);
            \draw [fill=white] (0,0) circle [radius=2pt];
        \end{tikzpicture}
    $$
    has a ray intersection $\lambda \mapsto (\lambda,0)$ but its closure is the sum $\{(x,y)\mid x=0\land 0\le y \le 1\}+ \{(x,y)\mid y=0\}$ of a bounded set and a vector space.
\end{rem}
\begin{lemma} \label{LemmaAltDescriptionInnerVspace}
    Let $S$ be the sum of a bounded set $B$ and a vector space $V$. Then, $V$ is the inner vector space of $S$.
\end{lemma}
\begin{proof}
    Recall that the definition of the inner vector space, Definition \ref{DefinitionInnerVspace} refers to $\closure{S}$ and not to $S$ directly.
    If $S=B+V$ and $v\in V$, then we get that $s+\lambda v \in S$ for all $s\in S$. 
    This implies $s+\lambda v\in \closure{S}$ for all $s\in \closure{S}$ and therefore $v$ is in the inner vector space of $S$. 
    
    On the other hand, if $\xx$ is in the inner vector space, then $s+\lambda \xx\in \closure{S}$ for all $\lambda\in \IR$ and $s\in S$. 
    By taking $s$ as essentially inner point, we get $s+\lambda \xx\in S$ by Lemma \ref{LemmaStrahlenVerschieben}. 
    Assume to the contrary that $\xx\notin V$ and write $\xx=u+\vv$ with $\vv\in V$ and $u\ne 0$ orthogonal to $V$. 
    Since $B$ is bounded, there is $R>0$ such that $B\subset \ball_R(0)$. Take $\lambda=\frac{2R}{\|u\|}$. 
    Then $s$ and $s+\lambda \xx$ are in $S$ and they can be written as sum of elements in $V$ and $B$ as $s=\vv_1+b_1$ and $s+\lambda \xx = \vv_2+b_2$ with $\vv_1,\vv_2\in V; b_1,b_2\in B$. Their difference is
    $$
        (s+\lambda \xx)-s=\lambda \xx= \frac{2R}{\|u\|}\xx = 2R\frac{1}{\|u\|}u + \frac{2R}{\|u\|}\vv
    $$
    respectively
    $$
        (s+\lambda \xx)-s= (\vv_2+b_2)-(\vv_1+b_1) = (\vv_2-\vv_1)+(b_2-b_1)
    $$
    when writing both as a sum.
    Recall that $u$ is orthogonal to all elements in $V$, including $\vv, \vv_2,\vv_1$. Therefore, the triangle $0$, $2R\frac{1}{\|u\|}u$, $(b_2-b_1)$ has a right angle at its second vertex. Thus we get the last inequality in
    $$
        2R>\|b_1\| + \|b_2\| \ge\|b_2-b_1\| \ge \|2R\frac{1}{\|u\|}u\| = 2R
    $$
    where the first inequality used that $b_1,b_2\in B\subset \ball_R(0)$. This is a contradiction. Therefore, $\xx\in V$ for all $\xx$ in the inner vector space.
\end{proof}

\subsection{Defining the ray}
This section focuses on infinitary pp-defining the ray in the cases where this is possible. There are many equivalent characterisations of this set which will be gathered together in Theorem \ref{TheoremFall1Ausgerollt}.

As this is the first time, we actually give an infinitary pp-definition, we should make a remark on our notation. The reason is that for example the sets $\{x\in \IR \mid 0<x\}$ and $\{x\in \IR \mid \lnot (x\le 0)\}$ are equal. However, they have two different definitions and only the first definition is a primitive positive definition because the second definition uses a negation.
\begin{notation}
    For any structure $\structA$ and any formula $\Phi$ over the same signature in $n$ free variables, we
    write that a set $S$ is given by
    $$
        \left\{ v \in \setstructA^n \mitte \Phi(v) \right\}
    $$
    to describe that $S=\{v \in \setstructA^n \mid \Phi(v)\}\subset \setstructA^n$ and moreover we consider the abstract formula $\Phi$ as definition of $S$. We use this notation instead of just $\Phi$ to explicitly write down the free variables and the dimension. %omit the curly brackets to stress that this is a definition of a set rather than a set.
\end{notation}

\begin{lemma} \label{LemmaInterDefinableHalveIntervalls}
    The open interval $\{x\mid x>0\}\subset \IR$ and the closed interval $\{x\mid x\ge 0\}\subset \IR$ are infinitary pp-interdefinable over $(\IR, +, c \mid c \in \IQ)$.
\end{lemma}
\begin{proof}
    The closed interval can be defined as infinite intersection as
    \begin{align*}
        \left\{ x \in \IR \mitte
        \bigwedge_{n \in \IN} x+\frac{1}{n} > 0 \right\}
    \end{align*}
    where $\frac{1}{n}$ is the rational constant.
    The open interval can be defined by the formula
    \begin{align*}
        \left\{ x \in \IR \mitte x\ge 0 \land \exists \var{1} \colon \bigwedge_{n \in \IN} \var{1}+(-n)+xn^2\ge 0 \right\}
    \end{align*}
    where $-n$ is the constant in $\IQ$ and multiplication with $n^2$ is implemented as repeated addition. It is easy to verify that the formula works. The idea behind the definition is to take one component of a hyperbola and project it on the $x$-axis.
\end{proof}

The next lemma gives sufficient conditions for defining the ray. As we will see in Theorem \ref{TheoremFall1Ausgerollt}, this will be equivalent conditions.
\begin{lemma} \label{LemmaImplicationsToRay}
    Let $S\subset \IR^n$ be a convex space. Then each item in the following list implies the next one:
    \begin{enumerate}
        \item \label{NrLemmaImplicationsToRay1}
        There is an affine set $A$, such that the closure $\closure{A\cap S}$ is not the sum of a bounded set and a vector space.
        \item \label{NrLemmaImplicationsToRay2}
        The set $S$ has a ray intersection.
        \item \label{NrLemmaImplicationsToRay3}
        At least one of the sets $\{x\in \IR \mid x>0\}$ and $\{x\in \IR \mid x\ge 0\}$ is pp-definable over $(\IR, +, \cdot c\mid c \in \IR, c \mid c \in \IR, S)$.
        \item \label{NrLemmaImplicationsToRay4}
        Both sets $\{x\in \IR \mid x>0\}$ and $\{x\in \IR \mid x\ge 0\}$ are infinitary pp-definable over
        $(\IR, +, \cdot c\mid c \in \IR, c \mid c \in \IR, S)$.
    \end{enumerate}
\end{lemma}
\begin{proof}
    \ref{NrLemmaImplicationsToRay1} implies \ref{NrLemmaImplicationsToRay2} is shown in Theorem \ref{TheoremClosureImpliesRay}.
    
    \ref{NrLemmaImplicationsToRay2} implies \ref{NrLemmaImplicationsToRay3} because there are $v,w\in \IR^n$ such that $v+\lambda w\in {S}$ for all $\lambda>0$ and $v+\lambda w\notin {S}$ for all $\lambda <0$ by \ref{NrLemmaImplicationsToRay2}. Now, the definition of $X\subset \IR$ as 
    $$
        \{x \in \IR \mid (v_1+w_1x, \dots , v_n+w_nx)\in S\}
    $$
    where $(v_1,\dots,v_n)^T=v$ and $(w_1,\dots,w_n)^T=w$ is clearly a pp definition over
    $(\IR, +, \cdot c\mid c \in \IR, c \mid c \in \IR)$. This set is the set of positive reals if $v\notin S$ and the set of non-negative reals otherwise.
    
    \ref{NrLemmaImplicationsToRay3} implies \ref{NrLemmaImplicationsToRay4} because those sets are infinitary  pp-interdefinable by Lemma \ref{LemmaInterDefinableHalveIntervalls}.
\end{proof}

%-------------
\subsection{The affine case}
Every affine space is definable from $(\IR, +, \cdot c\mid c \in \IR, c \mid c \in \IR)$. Since we will always have those operations, it is worth to prove this fact separately.
\begin{notation}
    For $u,v\in \IR^n$, $\lambda\in \IR$ , $S\subset \IR^n$ we write as usual $u+v$ for the vector sum $(u_1+v_1,\dots ,u_n+v_n)$, $\lambda u$ for the scalar product $(\lambda u_1,\dots ,\lambda u_n)$ and $u\in S$ if the vector is in the set. Note that the formulas behind this notation are primitive positive, so for example $\lambda u + v\in S$ is pp-definable.
\end{notation}
\begin{lemma} \label{LemmaSkalarProductDefinable}
    Let $v\in \IR^n$. Then, the set $\{(u,\lambda)\in \IR^n\times \IR \mid \langle u, v \rangle = \lambda\}$ is pp-definable over $(\IR;+,\cdot \lambda \mid \lambda \in \IR, \lambda \mid \lambda \in \IR)$.
\end{lemma}
\begin{proof}
    The description
    $$
        \{(u_1,\dots ,u_n,\lambda) \in \IR^n\times \IR \mid v_1u_1+\dots +v_nu_n=\lambda\}
    $$
    pp-defines the set $\{(u,\lambda)\mid \langle u, v \rangle = \lambda\}$.
\end{proof}
\begin{thm} \label{TheoremAffineFormNothing}
    Let $A\subset \IR^n$ be an affine subspace. Then, $A$ is pp-definable over $(\IR;+,\cdot \lambda \mid \lambda \in \IR, \lambda \mid \lambda \in \IR)$.
\end{thm}
\begin{proof}
    Since $A$ is affine, there is a vector space $V\subset \IR^n$ and a vector $w\in \IR^n$ such that $A=\{v+w\mid v \in V\}$. Let $\{b_1,\dots ,b_k\}$ be a basis of $V$. Then, $A$ can be defined as
    $$
        \{x\in \IR^n \mid \exists \lambda_1,\dots ,\lambda_k \in \IR :
        x=v+\lambda_1b_1+\dots +\lambda_kb_k\}
    $$
    where $x=v+\lambda_1b_1+\dots +\lambda_kb_k$ is a shorthand notation for $n$ linear equations.
\end{proof}
\begin{lemma} \label{LemmaAffinemap}
    Let $f \colon \IR^n \to \IR^n$ be an invertible affine map. Let $\sigma$ be a signature on $\IR$ that includes $(\IR, +, \cdot c\mid c \in \IR, c \mid c \in \IR)$. Then $S\subset \IR^n$ is pp-definable in $\sigma$ if and only if $f(S)$ is pp-definable.
\end{lemma}
\begin{proof}
    Write $f(v)=Av+b$.
    The set $f(S)$ is given by
    $$
        \{x=(x_1,\dots ,x_n)\in \IR^n \mid \exists y=(y_1,\dots ,y_n) \in\IR^n : y\in S \land Ay + b = x\}
    $$
    which is a pp-formula using vector notation assuming $S$ is pp-definable. Conversely, if $f$ is invertible and $f(S)$ is pp definable, then the pp-formula in
    $$
        \{x=(x_1,\dots ,x_n)\in \IR^n \mid Ax+b\in f(S) \}
    $$
    defines $S$.
\end{proof}
\subsection{Defining the representative in the bounded cases}
    For this subsection, we focus on convex sets $S$ which can be written as a sum of a bounded set and a vector space.

\begin{lemma} \label{LemmaDefiningOpenIntervall}
    Let $S\subset \IR^n$ be a convex set which is not closed.
    Then, the open interval $(0,1)$ is pp-definable from $(\IR;+,\cdot c \mid c\in \IR, c\mid c\in \IR, S)$.
\end{lemma}
\begin{proof}
    Fix any essentially inner $\inner \in S$ which exists by Corollary \ref{CorExistInnerPoint} and any $s\in \closure{S}\setminus S$.
    
    By Lemma \ref{LemmaInnenGegenRand}, the points $\lambda \inner + (1-\lambda) s$ are in $S$ for $\lambda \in (0,1)$. Since $S$ is convex and $s\notin S$, the points given by $\lambda \inner + (1-\lambda) s$ for $\lambda \le 0$ are not in $S$. Therefore, the pp-definition
    $$
        \{x \in \IR \mid x (\inner -s) + s \in S \land (1-x) (\inner -s) + s \in S\}
    $$
    defines the open interval $(0,1)$. Here, $u$ and $s$ denote a vector in $\IR^n$ with the usual addition and scalar multiplication.
\end{proof}
\begin{rem}
    Recall that a set $S\subset \IR^n$ is compact if and only if it is closed and bounded by the Heine-Borel theorem \cite{BorelHeine}. Recall furthermore that the product of compact sets is compact by Tychonoff's theorem \cite{Tychonoff} and the image of a compact set under a continuous map is compact.
\end{rem}
\begin{lemma} \label{LemmaClosedEqualsCompact}
    Let $S\subset \IR^n$ be a convex set which is the sum of a bounded set $B$ and a vector space $V$.
    Then, the following are equivalent:
    \begin{enumerate}
        \item \label{ConditionLemmaClosedEqualsCompact1}
        The set $S\subset \IR^n$ is closed.
        \item \label{ConditionLemmaClosedEqualsCompact2}
        There is a compact set $C$ and $S=C+V$.
    \end{enumerate}
\end{lemma}
\begin{proof}
    \ref{ConditionLemmaClosedEqualsCompact1} implies \ref{ConditionLemmaClosedEqualsCompact2}: Let $C$ be the closure of $B$. Since $S$ is closed, we get $C\subset S$ and thus $C+V\subset S+V=S$. On the other hand, we have $S=B+V\subset C+V$ which shows the other inclusion.
    
    \ref{ConditionLemmaClosedEqualsCompact2} implies \ref{ConditionLemmaClosedEqualsCompact1}: It suffices to show that $S\cap \closedball_r(0) \subset \IR^n$ is closed for all $r\in \IN$. Choose $r\in \IN$ and choose any $R>0$ such that the compact set $C$ is contained in $\ball_R(0)$. Assume that $s\in S\cap \closedball_r(0)$. Then, we can write $s=c+v$ with $c\in C$, $v\in V$. Since $\|c\|<R, \|s\|\le r$, we get $\|v\|<r+R$. So we get
    $$
        S\cap \closedball_r(0) = (B + \{v\in V \mid \|v\|\le r+R\}) \cap \closedball_r(0)
    $$
    as equality of sets. Since $B$ is compact and $\{v\in V \mid \|v\|\le r+R\}$ is closed and bounded and thus compact, also their sum is compact and thus closed. Therefore $S\cap \closedball_r(0)$ is closed for all $r$ and thus $S$ is closed.
\end{proof}

\begin{lemma} \label{LemmaDefiningCompactIntervall}
    Let $S$ be a non-affine convex set which is the sum of a bounded set $C$ and a vector space $V$. Then, the compact interval $[0,1]$ is infinitary pp-definable from $(\IR;+,\cdot c \mid c\in \IR, c\mid c\in \IR, S)$.
\end{lemma}
\begin{proof}
    If $S$ is not closed, the open interval is definable by Lemma \ref{LemmaDefiningOpenIntervall} and the description
    $$
        \left\{ x \in \IR \mitte \bigwedge_{n\in \IN} \frac{1}{2} + (1-\frac{1}{n})(x-\frac{1}{2}) \in (0,1) \right\}
    $$
    defines $[0;1]$ using the constant $\frac{1}{2}$ and the scalar $1-\frac{1}{n}$.

    Now, assume that $S$ is closed. By Lemma \ref{LemmaClosedEqualsCompact}, we may assume that $C$ is compact. Note that $V$ is the inner vector space by Lemma \ref{LemmaAltDescriptionInnerVspace}.

    Take any $v\notin V$ and look at the affine set $A \coloneqq \{s+\lambda v\mid \lambda \in \IR \}$ for an essentially inner point $s\in S$. By Lemma \ref{LemmaMaximalVectorforClosed}, there is $\lambda \in \IR$ such that $s+\lambda v \notin S$. Since $S$ is not affine, we even get that $S$ is different from its affine hull and may assume that $s+\lambda v$ is in it. Moreover, the closed convex set $S\cap A$ can be understood as an interval inside $A$. Since it comes from intersecting with a bounded set plus vector space, it cannot be a ray. As it is not everything, $S\cap A$ is a compact interval. It cannot be a point, because $s$ was an essentially inner point. So there are $u,w \in S\cap A$ such that $S\cap A$ is the interval from $u$ to $w$.
    In other words,
    $$
        \{x \in \IR \mid u+x(w-u)\in S\}
    $$
    pp-defines $[0;1]$. This definition uses $u$ as a constant and $w-u$ as a scalar constant. Note that this definition is a notationally shorted $n$-dimensional formula.
\end{proof}

\begin{lemma} \label{LemmaDefiningPointedRectangle}
    Let $S\subset \IR^n$ be a convex set which is not closed.
    Assume furthermore that $S=B+V$ for a bounded set $B$ and a vector space $V$ and assume that there is a \dent{} $a\notin S$.
    
    Then, $\{(x,y) \in \IR^2 \mid (x\in (-1,1)\land y \in (0,1)) \lor (x,y)=(0,0)\}$ is infinitary pp-definable from $(\IR;+,\cdot c \mid c\in \IR, c\mid c\in \IR, S)$.
\end{lemma}
\begin{proof}
    For $n\le 1$, there is no convex set with a \dent{}.
    
    For $n\ge 2$, write $a=\lambda t+(1-\lambda)s$ with $s\in S, t\in \closure{S}, \lambda\in (0,1)$. 
    We identify $\IR^n$ with $\IR \times \IR \times \IR^{n-2}$ and may assume that $s=(0,0,0)$ and $t=(0,1,0)$ by applying an invertible affine map. This is allowed by Lemma \ref{LemmaAffinemap}. Then, $a=(0,\lambda,0)$. Since $S$ is convex, the points $(0,\mu,0)$ are not in $S$ for $\mu\ge \lambda$. 
    But $(0,1,0)$ is an accumulation point, so there is a point $b= (b_1,b_2,b') \in S \cap \ball_{1-\lambda}(t)$. By Corollary \ref{CorInnerPointsDense} there is furthermore an essentially inner point $b$ in this set. Since, $b$ is in $S$, we get that $b_1\ne 0$ or $b'\ne 0$.
    
    As $\closure{S}$ is convex, all points in the line segment between $s$ and $t$ are in $\closure{S}$. Then, Lemma \ref{LemmaInnenGegenRand} implies that all point in the interior of the triangle $s,t,b$ are in $S$.
    For simplicity make an affine map to move $b$ to $(1,0,0)$ (Lemma \ref{LemmaAffinemap}). 
    $$
    \begin{tikzpicture}
        \fill [gray!20!white] (0,0) -- (1,0) -- (0,1) -- (0,0);
        \draw (0,0) -- (1,0) -- (0,1);
        \draw [dashed] (0,0.6) -- (0,1);
        \draw (0,0) node [anchor=east]{$s$};
        \draw (0,1) node [anchor=east]{$t$};
        \draw (1,0) node [anchor=west]{$b$};
        \draw (0,0.6) node [anchor=east]{$a$}; 
        \draw [fill=white] (0,0.6) circle [radius=2pt]; 
        \draw [fill=white] (0,1) circle [radius=2pt]; 
        \draw [fill] (0,0) circle [radius=2pt]; 
        \draw [fill] (1,0) circle [radius=2pt]; 
    \end{tikzpicture}
    $$
    Note that $[0,1]$ is infinitary pp-definable by Lemma \ref{LemmaDefiningCompactIntervall}. Now, define $S_1$ as
    $$
        \{(x_1,x_2)\in \IR^2 \mid (x_1,x_2,0)\in S \land x_1\in [0,1] \land x_2\in [0,1] \land x_1+x_2\in [0,1]\}
    $$
    Then, $S_1$ containes the triangle $\{(x_1,x_2) \mid x_1>0,x_2\ge0 , x_1+x_2\le1\}$, is contained in the closure of that triangle and containes the point $(0,0)$ but excludes the point $(0,\lambda)$. Let $c\le \lambda$ be the smallest nonnegative number such that all $(0,\mu)$ are in $S_1$ for $0\le \mu<c$ and not in $S_1$ for $\mu>c$.
    Define $S_2$ as
    $$
        \left\{ (x_1,x_2) \in \IR^2 \mitte \bigwedge_{n\in \IN} \exists y\in \IR : ny\in [0,1]\land (x_1,x_2-y)\in S_1. \right\}
    $$
    This set contains $S_1$. Note that now, $(0,c) \in S_2$. 
    $$
    \begin{tikzpicture}
        \fill [gray!20!white] (0,0) -- (1,0) -- (0,1) -- (0,0);
        \draw (0,0.4) -- (0,0) -- (1,0) -- (0,1);
        \draw [dashed] (0,0.4) -- (0,1);
        \draw [dotted] (0,0.4) -- (1,0);
        \draw (0,0) node [anchor=east]{$s$};
        \draw (0,1) node [anchor=east]{$t$};
        \draw (1,0) node [anchor=west]{$b$};
        \draw (0,0.4) node [anchor=east]{$(0,c)$}; 
        \draw [fill] (0,0.4) circle [radius=2pt]; 
        \draw [fill=white] (0,1) circle [radius=2pt]; 
        \draw [fill] (0,0) circle [radius=2pt]; 
        \draw [fill] (1,0) circle [radius=2pt]; 
        %\clip (-1,-0,2) rectangle (2,1.3);
    \end{tikzpicture}
    $$
    Define $S_3$ as
    $$
        \{(x_1,x_2) \in \IR^2 \mid (x_1, (1-c) x_2 + c -c x_1)\in S_2 \land x_2 \in [0,1].\}
    $$
    This is again a subset of the closed triangle $\{(x_1,x_2) \mid x_1,x_2\ge 0, x_1+x_2\le 1\}$ containing the interior of this triangle and this time, it contains $(0,x_2)$ if and only if $x_2=0$.
    $$
    \begin{tikzpicture}
        \fill [gray!20!white] (0,0) -- (1,0) -- (0,1) -- (0,0);
        \draw (0,0) -- (1,0) -- (0,1);
        \draw [dashed] (0,0) -- (0,1);
        \draw [dotted] (0,0) -- (0,-0.6) -- (1,0);
        \draw [dotted] (0.5,0) -- (0.5,0.5) -- (0,0.5);
        \draw (0,0) node [anchor=east]{$(0,0)$};
        \draw (0,1) node [anchor=east]{$(0,1)$};
        \draw (1,0) node [anchor=west]{$(1,0)$};
        \draw [fill=white] (0,1) circle [radius=2pt]; 
        \draw [fill] (0,0) circle [radius=2pt]; 
        %\draw [fill] (1,0) circle [radius=2pt]; 
    \end{tikzpicture}
    $$
    To make it a rectangle define $S_4$ as
    $$
        \left\{ (x_1,x_2) \in \IR^2 \mitte  (\frac{1}{2} x_1,\frac{1}{2} x_2) \in S_3 \land x_1 \in [0,1] \land x_2\in [0,1] \right\}
    $$
    using Lemma \ref{LemmaDefiningCompactIntervall} again. This is now a subset of the square $\{(x_1,x_2 \mid 0\le x_1,x_2\le 1)\}$.
    $$
    \begin{tikzpicture}
        \fill [gray!20!white] (0,0) -- (1,0) -- (1,1) -- (0,1) -- (0,0);
        \draw (0,0) -- (1,0) -- (1,1) -- (0,1);
        \draw [dashed] (0,0) -- (0,1);
        \draw [dotted] (1,0) -- (2,0) -- (0,2) -- (0,1);
        \draw (0,0) node [anchor=east]{$(0,0)$};
        \draw (0,1) node [anchor=east]{$(0,1)$};
        \draw (1,0) node [anchor= west, fill=white]{$(1,0)$};
        \draw (1,1) node [anchor= west]{$(1,1)$};
        \draw [fill=white] (0,1) circle [radius=2pt]; 
        \draw [fill] (0,0) circle [radius=2pt]; 
    \end{tikzpicture}
    $$
    Finally define $S_5$ as
    \begin{align*}
        \{(x_1,x_2)\in \IR^2 \mid \exists y&: (x_1,y) \in S_4 \land (x_1, y-x_2) \in S_4
        \\ & \land 2x_2-1\in (0,1) \land \frac{1}{2} x_1 + \frac{1}{2} \in (0,1)\}
    \end{align*}
    using Lemma \ref{LemmaDefiningOpenIntervall} to define $(0,1)$. The first line ensures that $(x_1,x_2)$ is in the rectangle $\{(x_1,x_2) \mid 0\le x_1 \le 1 \land -1 \le x_2\le 1\}$ and if $x_1=0$ then $x_2=0$. The second line excludes the other boundaries, so $S_5$ is 
    $$
        \begin{tikzpicture}
            \fill [gray!20!white] (-1,1) -- (1,1) -- (1,0) -- (-1,0) -- (-1,1);
            \draw [very thin, dashed] (-1,1) -- (1,1);
            \draw [very thin, dashed] (-1,0) -- (1,0);
            \draw [very thin, dashed] (-1,0) -- (-1,1);
            \draw [very thin, dashed] (1,0) -- (1,1);
            \draw [fill] (0,0) circle [radius=2pt];
            \draw (-1,0) node [anchor=east]{$(-1,0)$};
            \draw (-1,1) node [anchor=east]{$(-1,1)$};
            \draw (1,0) node [anchor= west]{$(1,0)$};
            \draw (1,1) node [anchor= west]{$(1,1)$};
        \end{tikzpicture}
        $$
        the desired set.
\end{proof}

\subsection{Defining the pointed stripe}
It is left to define the representative in the case where the given set is not the sum of a bounded set and a vector space but it cannot define the ray.
\begin{lemma} \label{LemmaDefiningPointedStripeLimited}
    Let $S$ be a convex subset of $\IR^n$. Assume that $S$ is not the sum of a bounded set and a vector space. Assume furthermore that for each affine space $A$, the closure $\closure{S\cap A}$ is the sum of a vector space and a compact set. Then $\{(x,y)\in \IR^2 \mid y \in (0,1) \lor (x,y)=(0,0)\}$ is infinitary pp-definable from $(\IR;+,\cdot c \mid c\in \IR, c\mid c\in \IR, S)$.
\end{lemma}
\begin{proof}
    By the assumption with the affine set $A=\IR^n$, we get that $\closure{S}=C+V$ for a compact set $C$ and a vector space $V$. Since $S$ cannot be written as the sum of the bounded set $C\cap S$ and the vector space $V$, there is $s\in S$ and $v \in V$ such that $v+s\notin S$. Fix $s,v$ and an essentially inner point $\inner$ and define the set $S'$ as
    $$
        \{(x,y)\in \IR^2 \mid xv + y(s-\inner) + \inner \in S\}
    $$
    where vector notation is used.
    $$
    \begin{tikzpicture}
        \draw [dotted, fill=gray!20!white] (0,0) circle (0.4);
        \draw [fill=white] (1,1) circle [radius=2pt]; 
        \draw [fill] (0,1) circle [radius=2pt]; 
        \draw [fill] (0,0) circle [radius=2pt]; 
        \draw (0,0) node [anchor=east]{$(0,0)$};
        \draw (0,1) node [anchor=east]{$(0,1)$};
        \draw (1,1) node [anchor= west]{$(1,1)$};
    \end{tikzpicture}
    $$
    Note that $(0,0)\in S', (0,1)\in S', (1,1)\notin S'$. Moreover, Lemma \ref{LemmaStrahlenVerschieben} implies that $\inner + \lambda v \in S$ for all $\lambda \in \IR$, so $(\lambda,0)\in S'$ for all $\lambda\in \IR$.
    $$
        \begin{tikzpicture}
            \fill [gray!20!white] (-5,1) -- (5,1) -- (5,0) -- (-5,0) -- (-5,1);
            %\draw [fill] (0,0) circle [radius=2pt];
            \draw [fill=white] (0.6,1) circle [radius=2pt]; 
            \draw [fill] (0,1) circle [radius=2pt]; 
            \draw [fill] (0,0.4) circle [radius=2pt]; 
            \draw (0,0.4) node [anchor=east]{$(0,0)$};
            \draw (0,1) node [anchor=east]{$(0,1)$};
            \draw (0.6,1) node [anchor= west]{$(1,1)$};
        \end{tikzpicture}
    $$
    
    If we choose $A$ in the assumption as the two dimensional affine space containing $s,\inner, s+v$, then we get that the closure of $S'$ is the sum of a vector space $V'$ and a compact set $C'$. Since $v\in V$, the point $(1,0)$ is in $V'$. Since $V'\ne \IR^2$, we get that this is already all of $V'$. Therefore, there is a minimal $t$ such that $(0,t)\in \closure{S}$. Since then also $(\lambda, t)\in \closure{S}$ for all $ \lambda$, Lemma \ref{LemmaInnenGegenRand} shows that $(\lambda, \mu)\in S'$ for all $\lambda \in \IR$ and $\mu \in (t,1)$. Since $t$ was minimal, $(1,1)\notin S'$ and the same argument would apply, we get that $(\lambda, \mu)\notin S'$ for all $\mu<t$ or $\mu>1$.
    
    Finally choose $A$ in the assumption as the affine line containing $s$ and $s+v$. Then, $S$ intersected with this line is not everything so there are $a,b\in \IR$ such that $\closure{S\cap A}$ is represented as the compact intervall from $a$ to $b$. In other words, $\lambda \in [a,b]$ implies that $(\lambda,1)\in \closure{S'\cap (\IR \times \{1\})}$.
    $$
        \begin{tikzpicture}
            \fill [gray!20!white] (-5,1) -- (5,1) -- (5,0) -- (-5,0) -- (-5,1);
            %\draw [fill] (0,0) circle [radius=2pt];
            \draw [very thin, dashed] (-5,1) -- (-1,1);
            \draw [very thin, dashed] (0.6,1) -- (5,1);
            \draw [very thin] (-1,1) -- (0.6,1);
            \draw [very thin, dashed] (0,0) -- (0,-0.4);
            \draw [very thin, dashed] (0,1) -- (0,1.4);
            \draw [fill=white] (1.5,1) circle [radius=2pt]; 
            \draw [fill] (0,1) circle [radius=2pt]; 
            \draw [fill] (0,0.4) circle [radius=2pt]; 
            \draw (0,0.4) node [anchor=east]{$(0,0)$};
            \draw (-1,1) node [anchor=south]{$(a,1)$};
            \draw (0.6,1) node [anchor=south]{$(b,1)$};
            \draw (1.5,1) node [anchor=north]{$(1,1)$};
            \draw (0,0) node [anchor=west]{$(0,t) $};
        \end{tikzpicture}
    $$
    
    Consider the set defined as
    \begin{align*}
        \left\{ (x,y) \in \IR^2 \mitte (x,y) \in S' \land (x+a-b-1,y) \in S' \land
        \bigwedge_{n \in \IN} (nx, 1-y ) \in S'.\right\}
    \end{align*}
    This set contains no points with $y>1$ by the first condition and no points with $y<0$ by any of the last conditions. For $y=1$ the first or the second condition contradicts since  $x$ would need to be in the interval $[a,b]$ and in the interval $[b+1, 2b+1-a]$. For $y=0$ and $x\ne 0$, the term $nx$ will become smaller that $a$ or greater than $b$ so also those elements are not in the new set. On the other side both $(0,0)$ and $(x,y)$ with $0<y<1$ are easily seen to be in this set. 
        $$
        \begin{tikzpicture}
            \fill [gray!20!white] (-5,1) -- (5,1) -- (5,0) -- (-5,0) -- (-5,1);
            \draw [very thin, dashed] (-5,1) -- (5,1);
            \draw [very thin, dashed] (-5,0) -- (5,0);
            \draw [fill] (0,0) circle [radius=2pt];
        \end{tikzpicture}
        $$
    Therefore, this set defines the pointed stripe.
\end{proof}

%% file: 04Reverse-Construction.tex
\section{Defining arbitrary convex sets}
\label{SectionProofReverse}
In this section, we show how to infinitary pp-define any convex set $S$ over $(\IR;+,\cdot c \mid c\in \IR, c\mid c\in \IR, R)$, where $R$ is the representative of the equivalence class containing $S$.

\subsection{Preliminary}
The whole section of the reverse construction will commonly use versions of the 
Hahn-Banach-Theorem. We will not use the original versions by Hahn \cite{Hahn} and Banach \cite{ BanachReprint}. Instead, we use the following geometric versions:

\begin{definition}
    For a set $S$ denote by $\inside{S}$ the interior of $S$.
\end{definition}

\begin{thm}[Hahn-Banach, {\cite{MR1921556}}] \label{Theorem Hahn-Banach} \label{Theorem Hahn-Banach3-Boundary}
    Let $S\subset \IR^n$ be a convex subset and let $a\in \IR^n\setminus \inside{S}$. Then, there is $v\in \IR^n\setminus \{0\}$ of norm $1$ such that $S$ is contained in the closed halve space $\{h\in \IR^n \mid \langle h-a,v\rangle \ge 0\}$.
\end{thm}

There is also a version for closed sets:
\begin{thm}[Hahn-Banach, {\cite[Theorem IV.3']{MR1336382}}] \label{Theorem Hahn-Banach2}
    Let $S\subset \IR^n$ be a closed convex subset and let $a\in \IR^n\setminus S$. Then, there is $v\in \IR^n\setminus \{0\}$ of norm $1$ such that $S$ is contained in the open halve space $\{h\in \IR^n \mid \langle h-a,v\rangle > 0\}$. Moreover, there is $\epsilon>0$, such that $S\subset \{h\in \IR^n \mid \langle h-a,v\rangle \ge \epsilon\}$.
\end{thm}
The Theorems are stated in the way in which we will use them.
In both references \cite{MR1921556, MR1336382}, the theorem is stated for the more general case of a locally convex topological real vector space. In \cite{MR1921556}, it is assumed that the interior of $S$ is nonempty. This assumption is not needed here, because if the outer dimension of $S$ is $n$, then it has an inner point. If the outer dimension of $S$ is smaller, then we can take $v$ orthogonal to the affine hull of $S$.

\begin{lemma} \label{LemmaInnerVectorOrthogonal}
    Let $S\subset \IR^n$ be a nonempty convex subset and let $v$ be in the inner vector space of $S$ (see Definition \ref{DefinitionInnerVspace}). Let $a\in \IR^n\setminus S$. Let $w\in \IR^n$ be of norm $1$ such that $S$ is contained in the closed halve space $\{h\in \IR^n \mid \langle h-a,w\rangle \ge 0\}$. Then, $\langle v,w\rangle =0$.
\end{lemma}
\begin{proof}
    Let $s\in S$ be an essentially inner point (see Definition \ref{DefinitionEssInnerPoint} and Corollary \ref{CorExistInnerPoint}). Then, $\lambda v+s\in S$ for all $\lambda \in \IR$. If $\langle v,w\rangle$ would be nonzero, we could insert $\lambda= \frac{-\langle s-a, w \rangle - 1}{\langle v, w \rangle}$ and get
    $$
        0\le \langle (\lambda v+s)-a, w \rangle 
        = \lambda \langle v, w \rangle + \langle s-a, w \rangle
        = \frac{- \langle s-a, w \rangle - 1}{\langle v, w \rangle} \langle v, w \rangle+ \langle s-a, w \rangle
        =-1
    $$
    which is a contradiction. 
\end{proof}

\subsection{Definitions from the ray}
Every convex set is definable from the ray. Closed sets are a bit easier to define, because they have no excluded boundary points. We prove both facts in this section.

\begin{notation}
    We use $<$ as the relation symbol for the relation $\{x\in \IR \mid x>0\}\subset \IR$ and $\le$ for the relation $\{x\in \IR \mid x\ge 0\}\subset \IR$. Note that the usual set
    $$
        \{(x,y)\mid x<y\} =\{(x,y) \mid \exists z : x+z=0 \land y+z>0\}
    $$
    is pp-definable over $(\IR;<,0,+)$ so multiple meanings are no problem.
\end{notation}

\begin{thm} \label{TheoremClosedDefinable}
    Let $S$ be a closed convex set. Then, $S$ is infinitary pp-definable over $(\IR;+,\cdot \lambda \mid \lambda \in \IR, \lambda \mid \lambda \in \IR, <)$.
\end{thm}
\begin{proof}
    For each $a\in \IR^n\setminus S$, there exists by Theorem \ref{Theorem Hahn-Banach2} a vector $v_a$ such that $\langle s-a,v_a\rangle > 0$ for all $s\in S$. Define $S'$ as the intersection of all those sets as
    \begin{align*}
    \left\{
        s=(s_1,\dots ,s_n) \in \IR^n \mitte \bigwedge_{a \in \IR^n\setminus S} \exists \lambda \in \IR:
        \lambda>0 \land \langle s-a,v_a\rangle=\lambda \right\}
    \end{align*}
    where $v_a$ is the constant depending on $a$ as described above and $\langle\bullet, v_a\rangle$ is definable by Lemma \ref{LemmaSkalarProductDefinable}.
    
    This set $S'$ equals $S$: By definition $S\subset S'$. On the other hand, for each $a\notin S$, we get that $a$ is not in the open halve plane defined using $a$ so $a\notin S'$ and $S'=S$.
\end{proof}
\begin{lemma} \label{LemmaKonvexesBeispiel}
    The set $\{(x,y,z) \mid xz\ge y^2, x\ge 0, z \ge 0\}$ is convex.
\end{lemma}
\begin{proof}
    Let $x_1z_1\ge y_1^2$ and $x_2z_2\ge y_2^2$ such that $x_1,x_2,z_1,z_2\ge 0$ and take $0<\lambda<1$. Then, we get that
    \begin{align*}
    &&
        \lambda(x_1,y_1,z_1) + (1-\lambda)(x_2,y_2,z_2) &\in \{(x,y,z) \mid xz\ge y^2, x\ge 0, z \ge 0\}
        \\
        \iff &&
        (\lambda x_1+ (1-\lambda) x_2)(\lambda z_1+ (1-\lambda) z_2) &\ge(\lambda y_1 + (1-\lambda)y_2 )^2 
        \\ \impliedby &&
        (\lambda x_1+ (1-\lambda) x_2)(\lambda z_1+ (1-\lambda) z_2) &\ge(\lambda \sqrt{x_1}\sqrt{z_1} + (1-\lambda)\sqrt{x_2}\sqrt{z_2} )^2 
        \\ \iff && 
        (\lambda x_1+ (1-\lambda) x_2)(\lambda z_1+ (1-\lambda) z_2) &- (\lambda \sqrt{x_1}\sqrt{z_1} + (1-\lambda)\sqrt{x_2}\sqrt{z_2} )^2 \ge 0
        \\ \iff &&
        \lambda(1-\lambda)(\sqrt{x_1}\sqrt{z_2}-\sqrt{x_2}\sqrt{z_1})^2&\ge 0
    \end{align*}
    where the implication uses the assumption.
\end{proof}
\begin{cor} \label{CorollaryPointedHalvePlaneDefinition}
The set $\{(x,y)\mid x=y=0 \lor x>0\}$ is infinitary pp-definable over $(\IR;+,\cdot \lambda \mid \lambda \in \IR, \lambda \mid \lambda \in \IR, <)$.
\end{cor}
\begin{proof}
    The set
    $$
        \{(x,y,z) \mid xz\ge y^2, x\ge 0, z \ge 0\}
    $$
    is convex by Lemma \ref{LemmaKonvexesBeispiel} and closed and thus infinitary pp-definable by Theorem \ref{TheoremClosedDefinable}.
    Therefore, we can define
    $$
        \{(x,y) \in \IR \times \IR \mid \exists z: (xz\ge y^2, x\ge 0, z \ge 0)\}
    $$
    which contains for $y=0$ all $x\ge 0$ and for $y\ne 0$ only $x>0$, because $x=0$ implies the contradiction $0=xz\ge y^2>0$. Therefore, this is an infinitary pp-definition of the set $\{(x,y)\mid x=y=0 \lor x>0\}$.
\end{proof}
\begin{lemma} \label{LemmaPunktKonvexDefiniertTrennen}
    Let $S\subset \IR^n$ be convex and $a\in \IR^n\setminus S$. Then, there is a set $S_a$ which is infinitary pp-definable over $(\IR;+,\cdot \lambda \mid \lambda \in \IR, \lambda \mid \lambda \in \IR, <, \le)$ which contains all of $S$ but not $a$.
\end{lemma}
\begin{proof}
    We show this theorem by induction on the dimension $n$.
    For $n=0$, this is clear and $S=S_a=\emptyset$. For $n=1$, the set $S_a$ can be defined as one of the two candidates
    \begin{align*}
        \{x \in \IR &\mid a<x\}
        \\
        \{x \in \IR &\mid x<a\}
    \end{align*}
    as a convex set $S$ cannot contain elements smaller and greater than $a$ without containing $a$.
    
    For $n\ge 2$, Theorem \ref{Theorem Hahn-Banach} gives us a vector $v$. By applying an invertible affine map, see Lemma \ref{LemmaAffinemap}, we may assume that $a=0$ and $v$ is the last basis vector $e_n$. We identify $\{h\in \IR^n \mid \langle h,v\rangle=0\}$ with $\IR^{n-1}$. By the induction hypothesis, there is a definable set $S_a^{n-1}\subset \IR^{n-1}$ such that $S\cap \IR^{n-1}\subset S_a^{n-1}$ and $a\notin S_a^{n-1}$.
    
    In the case where $S_a^{n-1}=\emptyset$, the set $S_a$ can be defined as
    $$
        \{(s_1,\dots ,s_n)\in \IR^n \mid s_n>0\}
    $$
    
    In the case where $S_a^{n-1}\ne \emptyset$ let $P(x,y)$ be the relation $x=y=0 \lor x>0$ which is infinitary pp-definable by Corollary \ref{CorollaryPointedHalvePlaneDefinition}. Define $S_a$ as
    \begin{align*}
        \{s=(s_1,\dots ,s_n) \in \IR^n \mid \ &\exists b,y\in \IR^{n-1} : b\in S^{n-1}_a
        \\&\land
        P(s_n,y_1)\land \dots  \land P(s_n,y_{n-1}) 
        \\&\land
        s_1=b_1+y_1 \land \dots  \land s_{n-1}=b_{n-1}+y_{n-1}\}
    \end{align*}
    Now, $S_a$ contains precisely all $s$ where $s_n=0$ and $(s_1,\dots ,s_{n-1})\in S_a^{n-1}$ and all $s$ with $s_n>0$.
\end{proof}

\begin{thm} \label{TheoremDefiningFromRay}
    Let $S\subset \IR^n$ be convex. Then, $S$ is infinitary pp-definable over $(\IR;+,\cdot \lambda \mid \lambda \in \IR, \lambda \mid \lambda \in \IR,\le)$.
\end{thm}
\begin{proof}
    By Lemma \ref{LemmaInterDefinableHalveIntervalls}, also $<$ is definable.
    By Lemma \ref{LemmaPunktKonvexDefiniertTrennen}, there is a definable set $S_a$ for each $a\in \IR^n\setminus S$ containing $S$ and avoiding $a$. Define $S$ as
    $$
        \left\{s \in \IR^n\mitte \bigwedge_{a \in \IR^n\setminus S} s \in S_a\right\}
    $$
    This set contains $S$ and no point in $a$ outside of $S$.
\end{proof}

\subsection{Definitions from the compact interval}
It is possible to define every sum of a vector space and a compact set from a compact interval. Since compact sets include their boundary points, this case is one of the easier cases.

\begin{lemma} \label{LemmaCompactIntervallDefinable}
    Every compact interval $[a;b]$ is pp-definable over $(\IR;+,\cdot \lambda \mid \lambda \in \IR, \lambda \mid \lambda \in \IR, [0,1])$.
\end{lemma}
\begin{proof}
    The pp-definition
    $$
        \left\{ x \in \IR \mitte \frac{1}{b-a}(x-a) \in [0,1]\right\}
    $$
    defines $[a;b]$ using the constant $-a$ and the scalar $\frac{1}{b-a}$ assuming $a\le b$. For $a>b$ the set $[a;b]$ is empty and thus definable by a contradiction.
\end{proof}

\begin{thm} \label{TheoremCompactDefinable}
    Every convex compact set $C\subset \IR^n$ is infinitary pp-definable over $(\IR;+,\cdot \lambda \mid \lambda \in \IR, \lambda \mid \lambda \in \IR, [0,1])$.
\end{thm}
\begin{proof}
    Let $a\notin C$. Since $C$ is compact, there is $R_a>0$ such that $C\subset B_{R_a}(a)$.
    By the Hahn-Banach-Theorem \ref{Theorem Hahn-Banach2}, there is $v_a\in \IR^n$ of norm 1 and $\epsilon_a>0$ such that $C\subset \{h \in \IR^n \mid \langle h-a,v_a \rangle\ge \epsilon_a\}$. 
    Since $\langle v_a,h-a\rangle\le \|h-a\|<R_a$, we get that $\langle v_a,h-a\rangle \in [\epsilon_a, R_a]$.
    So we define $C$ as
    $$
        \left\{ c \in \IR^n \mitte \bigwedge_{a \in \IR^n\setminus C} \exists \lambda \in \IR : \langle c-a, v_a \rangle = \lambda \land \lambda \in [\epsilon_a, R_a] \right\}
    $$
    using $\langle \bullet, v_a \rangle$ as defined in Lemma \ref{LemmaSkalarProductDefinable} and $[\epsilon_a, R_a]$ as defined in Lemma \ref{LemmaCompactIntervallDefinable}.
\end{proof}

\begin{thm} \label{TheoremCompactVectDefinable}
    Let $S$ be the sum $\{c+v \mid c \in C, v \in V\}$ of a convex compact set $C$ and a vector space $V$. Then, $S$ is infinitary pp-definable over $(\IR;+,\cdot \lambda \mid \lambda \in \IR, \lambda \mid \lambda \in \IR, [0,1])$.
\end{thm}
\begin{proof}
    By Theorem \ref{TheoremCompactDefinable}, $C$ is definable. By Theorem \ref{TheoremAffineFormNothing}, $V$ is definable. Now, the pp formula
    $$
        \left\{ s \in \IR^n \mitte \exists v,c \in \IR^n : v\in V\land c\in C\land \bigwedge_{i=1}^n v_i+c_i=s_i\right\}
    $$
    defines $S$.
\end{proof}

\subsection{Definitions from the bounded interval}
The bounded open interval is the standard non-compact bounded set. It is able to define all bounded sets without a \dent{}.

\begin{lemma} \label{LemmaSAForOpenIntervall}
    Let $S\subset \IR^n$ be the sum of a bounded space and a linear subspace and $a\notin S$ such that there is an $n-1$ dimensional affine subspace containing $a$ and disjoint to $S$.
    Then, there is a set $S_a$ containing $S$ but not $a$ which is infinitary pp-definable over $(\IR;+,\cdot \lambda \mid \lambda \in \IR, \lambda \mid \lambda \in \IR, (0,1))$.
    
    If such an affine space exists for all $a\notin S$, then $S$ is infinitary pp-definable over $(\IR;+,\cdot \lambda \mid \lambda \in \IR, \lambda \mid \lambda \in \IR, (0,1))$.
\end{lemma}
\begin{proof} We may assume $S\ne\emptyset$.
    Let $B$ be a bounded set and $V$ a vector space such that $S=B+V$. Note that $V$ is the inner vector space by Lemma \ref{LemmaAltDescriptionInnerVspace}. For $a\in \IR^n\setminus S$, choose any $n-1$ dimensional affine subspace $A_a$ as described. Since $S$ is convex, all of $S$ is at the same side of $A_a$, so there is a normal vector $v_a$ of norm 1 which is orthogonal to $A_a$ such that $\langle s-a,v_a\rangle>0$ for all $s\in S$. Let $R_a>0$ be such that $B$ is contained in $\ball_{R_a}(a)$. Define $S_a$ as
    $$
        \{x \in \IR^n \mid  \exists y \in \IR : 
        y \in (0,1) \land R_ay=\langle x-a,v_a\rangle \}
    $$
    and $S'$ as
    $$
        \left\{ x \in \IR^n \mitte \bigwedge_{a \in \IR^n\setminus S} x\in S_a\right\}
    $$
    if the affine space exists for all $a\notin S$.
    
    Now for $a\notin S$, we get that $x$ cannot be in $S_a$, because $R_ay=\langle x-a,v_a\rangle$ needs to be greater than $0$. So $S'\subset S$.
    
    On the other hand, for each $x=s\in S$, we can write $s=b+w$ with $b\in B, w\in V$. Note that $\langle w,v_a\rangle =0$ by Lemma \ref{LemmaInnerVectorOrthogonal}. So 
    $$
        \langle s-a,v_a\rangle = \langle b+w-a,v_a\rangle = \langle b-a,v_a\rangle \in (0,R_a)
    $$
    and $s$ will not be excluded. Therefore, $S\subset S_a$ and $S\subset S'$.
    
    This shows that $S_a$ has the desired property and with the assumption for all $a$, we get the equality $S=S'$ and $S$ is therefore definable.
\end{proof}
Note that this lemma for example shows that the compact interval $[0,1]$ is infinitary pp-definable.
For generalising this result to all sets $S$ without \dent{}s, we will need a common Corollary of Carathéodory's theorem.
\begin{thm}[Carathéodory's theorem {\cite{MR1511425}}] \label{Theorem Caratheodory}
Let $C\subset \IR^n$ be a compact subset.
    Every Point in the convex hull of $C$ can be written as a convex combination of $n+1$ points in $C$.
\end{thm}

\begin{cor} \label{CorollaryConvexHullCompact}
    The convex hull of a compact set $C \subset \IR^n$ is compact.
\end{cor}

%\begin{comment}
        \begin{proof}
            Let 
            $$
                \Delta \coloneqq (\lambda_1,\dots ,\lambda_{n+1}\in [0,1]^{n+1} \mid \lambda_1+\dots +\lambda_{n+1}=1)
            $$
            Now, $C$, $\Delta$, and their products are compact. Look at the continuous map
            \begin{align*}
                C^{n+1} \times \Delta &\to \IR^n \\
                (b_{1},b_{2},\dots ,b_{n+1}, \lambda_1,\dots ,\lambda_{n+1}) &\mapsto \sum_{i=1}^{n+1} \lambda_i b_i
            \end{align*}
            where $b_i$ denotes a vector in $C$ and $\lambda_i$ is in $[0,1]$ for each $i$. The image of this map is the convex hull of $C$ by Carathéodory's theorem.
            Since the image of a compact set is compact, the image of this map is compact.
        \end{proof}
%\end{comment}
\begin{notation}
    Denote by $\conv (X)$ the convex hull of a set $X\subset \IR^n$.
\end{notation}

This allows us to conclude the next lemma.
\begin{lemma}\label{LemmaConvexClosureCommute}
    Let $B\subset \IR^n$ a bounded set. Then, $\conv(\closure{B})=\closure{\conv(B)}$. %the convex hull of the closure $\closure{B}$ is the closure of the convex hull of $B$.
\end{lemma}
\begin{proof}
    The closure of the convex hull is a convex set by Lemma \ref{LemmaClosureStaysConvex}. It contains $\closure{B}$ and thus also contains the convex hull of $\closure{B}$.
    
    Conversely, the convex hull of $\closure{B}$ is closed and bounded by Corollary \ref{CorollaryConvexHullCompact}. Therefore it contains the convex hull of $B$ and the closure of the convex hull of $B$.
    As both inclusions are shown, we get the equality.
\end{proof}
\begin{lemma} \label{LemmaNoDentAlternative}
    Let $S\subset \IR^n$ be a convex set. Then, the following are equivalent:
    \begin{enumerate}
        \item \label{NrLemmaNoHoleAlternative1}
        The set $S$ has no \dent{}.
        \item \label{NrLemmaNoHoleAlternative2}
        For every affine subspace $A\subset \IR^n$, the set $A\cap S$ is empty or $\closure{A\cap S}=A\cap \closure{S}$.
    \end{enumerate}
\end{lemma}
\begin{proof}
    \ref{NrLemmaNoHoleAlternative1} implies \ref{NrLemmaNoHoleAlternative2}:
    If $A\cap S$ is empty, we are done. Otherwise, $\closure{A\cap S}\subset \closure{A}\cap \closure{S} = A\cap \closure{S}$ is clear. For the other inclusion choose $a\in A\cap\closure{S}$ and $b\in A\cap S$. Then take $b+\lambda(a-b)$ for $\lambda>0$, $\lambda\to 0$. Since these points are no \dent{}s, they are inside $S$. Thus, $b$ is in the closure of $A\cap S$.
    
    \ref{NrLemmaNoHoleAlternative2} implies \ref{NrLemmaNoHoleAlternative1}: If $a=\lambda s + (1-\lambda)t$ is a \dent{} with $s\in S, t\in \closure{S}\setminus S$, then take $A$ to be the affine line containing $s$ and $t$. Now, $t\in A\cap\closure{S}$, but $t\notin \closure{S\cap A}$ as there is $a$ in between $t$ and this set.
\end{proof}

\begin{lemma} \label{LemmaBoundedForOpenIntervall}
    Let $S\subset \IR^n$ be a bounded convex set which has no \dent{}.
    Then, $S$ is infinitary pp-definable over $(\IR;+,\cdot \lambda \mid \lambda \in \IR, \lambda \mid \lambda \in \IR, (0,1))$. %Here, $(0,1)$ denotes  the corresponding interval.
\end{lemma}
\begin{proof}
    Let $a\notin S$. We want to define a set $S_a$ such that $S\subset S_a$, but $a\notin S_a$. If there is any $v\in \IR^n$ of absolute value $1$ such that $\langle s-a, v\rangle>0$ for all $s\in S$ then we define $S_a$ by Lemma \ref{LemmaSAForOpenIntervall} using the affine subspace $\{h \in \IR^n \mid \langle h-a, v\rangle=0\}$.
    
    If there is no such $v$, then we 
    construct a set $T'$ which projects surjectively onto $S$ and for each preimage of $a$, there will be a definable set $T_z$, separating the preimage from $T'$.
    So the set $S_a$ can be defined as
    $$
        \left\{ x \in \IR^n \mitte \exists y \in \IR^n : \bigwedge_{z\in \IR^n} (x,y)\in T_z \right\}
    $$
    in the end.
    
    To construct $T'$, we do the following induction. We start with $k=1$.
    Choose $v_1$ of norm 1 such that $\langle s-a, v_1\rangle\ge 0$ for all $s\in S$. 
    This is possible by Theorem \ref{Theorem Hahn-Banach}. 
    After that, we look at the subspace $S_1\coloneqq \{s\in S\mid \langle v_1,s-a\rangle=0\}$.
    This set is nonempty by assumption. Increase $k$ by 1.
    %Note that the intersection $\closure{S}\cap \{s \in \IR^n \mid \langle s-a, v_1\rangle= 0\}$ is again convex and also the closure of $S_1$ by Lemma \ref{LemmaNoDentAlternative}. 
    %The point $a$ is no \dent{} of this set, because $S_1\subset S$. 
    
    For $k>1$, continue as follows:
    Define $v_k$ such that for each $s\in S_{k-1}$,
    %with the property that the vector $s-a$ is orthogonal to all $v_{1},\dots ,v_{k-1}$, 
    we have $\langle s-a , v_k \rangle\ge 0$ and moreover $\|v_k\|=1$ and $v_k$ is orthogonal to each other $v_j$ with $j< k$. Define $S_k\subset S_{k-1}$ as the points $s$ such that $s-a$ is orthogonal to $v_k$.
    If $S_k=\emptyset$, %\langle s-a , v_k \rangle > 0$ for all $s$ such that $s-a$ is orthogonal to $v_{1},\dots ,v_{k-1}$, (including the case that there is nothing left) 
    then stop and remember $k$. Otherwise increase $k$ by $1$ and repeat.
    
    Since $a\notin S$, the affine hull of $S_k$ has always a positive dimension or is empty. As the dimension of the affine hull of $S_k$ decreases, this algorithm terminates.
    To simplify terms, identify $a$ with 0 and $v_j$ with the $j$-th unit vector by applying an invertible isometric affine map (Lemma \ref{LemmaAffinemap}).
    Define $R>0$ to be a number such that $S\subset B_R(0)$.

    Consider the sets
    $$
        T'=\{(x,y)\in \IR^n \times \IR^n \mid x \in S\land y_1 = |x_1| \land\dots \land y_n=|x_n|
        \} \subset \IR^{2n}
    $$
    and $T=\conv(T')\subset \IR^{2n}$. Note that they depend on the redefined unit vectors.
    The set $T$ is by definition convex. 
    It is also bounded as $\|(x,y)\|=\|(x,x)\|<\sqrt{2}R$ for all $(x,y)\in T'$.
    We consider the map $\IR^{2n}\to \IR^n$ which is given by the projection on the first $n$ coordinates and look at the preimages of $a=0\in \IR^n$. They are given by $\{(0,z)\mid z \in \IR^n\}$. We identify $\IR^n \times \IR^n$ with $\IR^{2n}$ and give for each $z\in \IR^n$ a set $T_z\subset \IR^{2n}$ such that $(0,z)\notin T_z$ and $T'\subset T_z$. 
    There are three cases to consider:
    \begin{enumerate}
    \item
        If $z$ has a negative entry on an $i$-th component, then 
        $$
            \left\{ (x,y) \in \IR^n\times \IR^n \mitte \frac{1}{R-z_i}(y_i-z_i) \in (0,1)\right\}
        $$
        is a possible pp-definable set for $T_z$: For $(x,y)=(0,z)$, we get $0\in (0,1)$ which is false, thus $(0,z)$ is not in this set. For $(x,y)\in T'$, we get $\frac{y_i-z_i}{r-z_i} \in (0,1)$ which is true as $R>y_i\ge 0>z_i$, thus $T'\subset T_z$. 
        %As $T_z$ is convex, we get also $T=\conv(T')\subset \conv(T_z)=T_z$.
    \item
        If the first $k$ coordinates of $z$ are equal to zero, where $k$ is the number from above, then
        $$
            \left\{ (x,y) \in \IR^n\times \IR^n \mitte \frac{1}{kR}(y_1+\dots +y_k) \in (0,1) \right\}
        $$
        is a possible pp-definable set for $T_z$. Clearly $(0,z)$ is not in this set, as the sum of zeros is zero. On the other side, 
        %each element in $T$ is a convex combination of elements in $T'$.
        %As $(0,1)$ is also convex, it suffices to show that for each $(x,y)\in T'$ the property $\frac{1}{kR}(y_1+\dots +y_k) \in (0,1)$ holds. 
        take a tuple $(x,y)\in T'$. As $y$ is given by absolute values, we get $y_i\ge 0$ for all $i$. Since $|x_i|$ is bounded by $R$, we get additionally $y_i<R$. 
        Moreover, at least one of the $y_i$ with $i\le k$ is strictly positive because $(x,y)\notin \emptyset=S_k$. Therefore, the sum $y_1+\dots +y_k$ is inside $(0,kR)$ and $(x,y)\in T_z$.
    \item
        It remains the case where all $z_i$ are non negative and there is a smallest entry $j\le k$ such that $z_j>0$. In this case, we show that we can define $T_z$ as the closure of $T$.

        Recall that $\closure{T}=\conv(\closure{T'})$ by Lemma \ref{LemmaConvexClosureCommute}. Assume that $(0,z)\in \closure{T}$. In this case, we can write 
        $$
            (0,z)
            = \sum_{i=1}^{n+1} \lambda_i (x^i,y^i)
        $$
        where the $\lambda_i$ are positive real numbers with sum $1$ and the points $(x^i,y^i)$ are in $\closure{T'}$. (Recall that $x^i$ is used to denote a vector of vectors and not an exponential function.) For all $r<j$, we get that the sum $\sum_{i=1}^{n+1} \lambda_i y^i_r=z_r$ is zero and since each of the summands is non-negative, we get that $y^i_r=0$ for all $r<j$. 
        As $(x^i,y^i)\in \closure{T'}$, we also get that $|x^i_r|=y^i_r=0$ and therefore $x^i_r=0$ for those $r$. 

        For the index $j$, we have again that $\sum_i \lambda_i x_j^i=0$. But since $y_j=\sum_i \lambda_i |x^i_j|$ is nonzero, one of the summands $\lambda_i x^i_j$ is nonzero. Thus, we get by a version of the pigeonhole principle that there is an index $p$ such that $x^p_j$ is strictly negative. So we derived a point $(x^p,y^p)\in \closure{T}$ with the property that $x^p_1=\dots=x^p_{j-1}=0>x^p_j$. Note that $x^p\in \closure{S}$.

        Let $A$ be the subspace of $\IR^n$ given by vectors such that the first $j-1$ coordinates are zero. 
        Since $j<k$, we get that $A\cap S=S_{j-1}$ is nonzero and thus $x^p\in A\cap \closure{S}=\closure{A\cap S}$ by Lemma \ref{LemmaNoDentAlternative}. 
        But $A\cap S$ contains only points where the $j$-th coordinate is at least $0$ this also holds for the closure. That is a contradiction, because $x^p_j<0$. 
        
        Therefore, we showed that $(0,z)$ is not in the closure of $T$. Let in this case $T_z=\closure{T}$. This set can be infinitarily pp-defined, because the compact interval $[0,1]$ can be infinitarily pp-defined by Lemma \ref{LemmaSAForOpenIntervall} and the closed bounded convex set $\closure{T}$ by Theorem \ref{TheoremCompactDefinable}.
    \end{enumerate}
    Now, we can define the set $S_a$ as
    $$
        \left\{ x \in \IR^n \mitte \exists y \in \IR^n : \bigwedge_{z\in \IR^n} (x,y)\in T_z \right\}
    $$
    in this case. It contains $S$ but not $a$.
    
    Since this whole procedure can be done for every $a$, we can define $S$ as  the set
    $$
        \left\{ x \in \IR^n \mitte \bigwedge_{a\in \IR^n} x \in S_a\right\}
    $$
    which is an infinitary pp-definition and contains all $s\in S$ and no $a\notin S$.
\end{proof}

\begin{thm} \label{TheoremDefinitionFromOpenIntervall}
    Let $S\subset \IR^n$ be the sum of a bounded space and a linear subspace such that there is no \dent{}.
    Then, $S$ is infinitary pp-definable over $(\IR;+,\cdot \lambda \mid \lambda \in \IR, \lambda \mid \lambda \in \IR, (0,1))$. %Here, $(0,1)$ denotes  the corresponding interval.
\end{thm}
\begin{proof}
    Let $V$ be the inner vector space of $S$ and let $B$ be the set of points in $S$ which are orthogonal to each vector in $V$. Then, $B$ is bounded, has no \dent{} and $S=B+V$. Since $V$ is a vector space, it is pp-definable by Theorem \ref{TheoremAffineFormNothing} and $B$ is infinitary pp-definable by Lemma \ref{LemmaBoundedForOpenIntervall}, so $V$ is definable as
    $$
        \{x\in \IR^n \mid \exists b,v\in \IR^n : b+v=x\land b\in B \land v\in V\}
    $$
    which completes the proof.
\end{proof}

\subsection{Definitions from the pointed rectangle}
We prove in this section that all bounded sets can be primitively positively defined from the pointed rectangle. %This will be proven in this section.

\begin{lemma} \label{LemmaIntersectionStaysBounded}
    Let $S, A\subset \IR^n$ be the sum of a bounded set and a vector space each. Then, $A\cap S$ is also the sum of a bounded space and a vector space.
\end{lemma}
%\begin{rem}
    This includes the case where $A$ is an affine set and/or $S$ is bounded. The dimension of the three involved vector spaces might differ.
%\end{rem}
\begin{proofof}{Lemma \ref{LemmaIntersectionStaysBounded}}
    Let $S=B+V$ and $A=C+W$ where $B$ and $C$ are bounded and $V$ and $W$ are vector spaces. We may assume that all vectors in $B$ are orthogonal to $V$ by choosing in each $V$ orbit of $B$ to the element with the lowest distance to $0$. We may assume similar that all vectors in $C$ are orthogonal to $W$. Let $r>0$ be a bound such that both $B$ and $C$ are contained in $\ball_r(0)$.
    
    Let $V'$ and $W'$ be the orthogonal complements of $V\cap W$ in $V$ respectively $W$. We consider the case where both of them are nonzero. Let $\alpha$ be the smallest angle between nonzero vectors in $V'$ and $W'$. Note that $\alpha<90^\circ$ and $\alpha>0$ because there is no common vector in both spaces and $\IR^n$ has a finite dimension. Denote $R\coloneqq \frac{r}{\sin(\alpha /2)}\ge r$.
    
    Now, we want to show $A\cap S$ equals $(\ball_R(0)\cap A \cap S)+(V\cap W)$: The inclusion $(\ball_R(0)\cap A \cap S)+(V\cap W)\subset A\cap S$ is clear. For the other inclusion take any element $s$ in $A\cap S$. This can be written as $s=c+u+w$ with $c\in C, u\in V\cap W, w\in W'$. Note that $c+w$ is also in $A\cap S$. It suffices to show that $c+w\in \ball_R(0)$. Write $c+w=b+v$ with $b\in B, v\in V$. We get that $v\in V'$, because $b$, $c$ and $w$ are orthogonal to $V\cap W$.

    We get the following picture:
    $$
    \begin{tikzpicture}
    \coordinate (A) at (0,0);
    \coordinate (B) at (0,1);
    \coordinate (C) at (1,-1);
    \coordinate (D) at (3,1);
        \draw (0,0)  node [anchor=east]{$0$} -- (3,1);
        \draw pic [draw=white!50!black, fill=black!10, angle radius=9mm, "$\ge \alpha$"] {angle = B--D--C};
        \draw (0,0) -- (0,1);
        \draw (0,0) -- (1,-1);
        \draw (0,1)  node [anchor=east]{$B \ni b$} -- (3,1);
        \draw (1,-1) node [anchor=east]{$C \ni c$} -- (3,1);
        \draw (3,1)  node [anchor=west]{$b+v=c+w \in A\cap S$};
        \draw pic [draw, angle radius=4mm, "$\cdot$"] {angle = A--B--D};
        \draw pic [draw, angle radius=4mm, "$\cdot$"] {angle = D--C--A};
    \end{tikzpicture}
    $$
    Note that the picture may hide a third dimension. Since the arc between $v$ and $w$ is at least $\alpha$, the sum of the (positive) angles $\angle (0, b+v, b)$ and $\angle (0, c+w, c)$ is also at least $\alpha$, so one of them is bounded by $\alpha/2$ from below.
    Since both $b$ and $c$ have distance at most $r$ from the origin and a right angle, the length of $b+v$ is at most $r\cdot (\sin(\alpha/2))^{-1}=R$. Therefore, $b+v\in B_R(0)$ and $A\cap S$ equals $(\ball_R(0)\cap A \cap S)+(V\cap W)$.

    In the other case, $V'=\{0\}$ or $W'=\{0\}$. We assume without loss of generality that $V'=\{0\}$. 
    Then, $S\cap A$ and $S$ have the property that adding vectors from $V$ is still preserving the sets. Thus, $S\cap A$ is given by $V+(B\cap A)$, where $B\cap A$ is bounded as a subset of $B$. So in both cases, $A\cap S$ is the sum of a bounded set and a vector space. 
\end{proofof}
\begin{lemma} \label{LemmaIntervallFromBox}
    Every open interval $(a,b)$ is pp-definable over $(\IR;+,\cdot \lambda \mid \lambda \in \IR, \lambda \mid \lambda \in \IR, \rect)$. Here, $\rect\subset \IR^2$ denotes  the set $\{(x,y)\in \IR^2\mid (x\in (-1,1)\land y \in (0,1)) \lor (x,y)=(0,0)\}$.
\end{lemma}
\begin{proof}
    If $a\ge b$, then a contradiction suffices. Otherwise, $b>a$.
    Now,
    $$
        \left\{ x \in \IR \mitte (0.5,\frac{1}{b-a} x + \frac{-a}{b-a})\in \rect \right\}
    $$
    works.
\end{proof}
\begin{thm} \label{TheoremBoundedDefining}
    Let $S\subset \IR^n$ be the sum of a convex bounded space and a linear subspace.
    Then, $S$ is infinitary pp-definable over $(\IR;+,\cdot \lambda \mid \lambda \in \IR, \lambda \mid \lambda \in \IR, \rect)$. Here, $\rect\subset \IR^2$ denotes  the set $\{(x,y)\in \IR^2\mid (x\in (-1,1)\land y \in (0,1)) \lor (x,y)=(0,0)\}$.
\end{thm}
\begin{proof}
    We show the lemma by induction on $n$.
    
    If $n=0$, then $\IR^n$ is a single point. Both this point and not this point are definable.
    
    For $n>0$, choose a bounded set $B$ and a vector space $V$ such that $B+V=S$.
    Let $a\in \IR^n\setminus S$. 
    By Theorem \ref{Theorem Hahn-Banach}, there is $v_a$ of norm 1 such that $\langle s-a, v_a\rangle\ge 0$ for all $s\in S$. Recall that $v_a$ is orthogonal to all vectors in $V$. Since $B$ is bounded, there is $R_a>0$ such that $B \subset \ball_{R_a}(a)$.
    Distinguish two cases:
    
    If there is no element $s$ in $S$ with $\langle s-a, v_a\rangle=0$, define $S_a$ as
    $$
        \{x \in \IR^n \mid \langle x-a, v_a\rangle \in (0,R_a)\}
    $$
    using Lemma \ref{LemmaIntervallFromBox}. This is the first case.
    
    If there is an element $s$ in $S$ with $\langle s-a, v_a\rangle=0$, choose an orthogonal matrix $O_a$ such that $O_av_a$ is the $n$-th basis vector $e_n$. Now, the set 
    $$
        S_a^{n-1}\coloneqq \{x \in \IR^{n-1} \mid \exists s \in S : O_a(s-a)=(x_1,\dots ,x_{n-1},0)^T\}
    $$
    %\todo{Konvention zum Transponieren}
    is again the sum of a bounded space and a vector space since this is not changed by the invertible affine maps given by the multiplication with $O_a$ and the translation by $-a$. It is also not changed by the intersection by Lemma \ref{LemmaIntersectionStaysBounded}. Therefore, $S_a^{n-1}$ is infinitary pp-definable by the induction hypothesis. Define $S_a$ as
    \begin{align*}
        \{x \in \IR^n \mid \exists y \in \IR^n; b\in \IR; c,z \in \IR^{n-1} &: 
        O_a(x-a)=y \land z \in S_a^{n-1} 
        \\&\land \bigwedge_{i=1}^{n-1} \rect(c_i,b) 
        \land \bigwedge_{i=1}^{n-1} z_i=y_i+2R_ac_i 
        \\&\land y_n=R_a b\}
    \end{align*}
    in this case. We explain this shortly. This is the second case.
    
    Finally, define $S'$ as
    $$
        \left\{ x \in \IR^n \mitte \bigwedge_{a \in \IR^n \setminus S} S_a(x).\right\}
    $$
    
    Now, look at any $a\in \IR^n\setminus S$. This is not in $S_a$ which is clear in the first case. In the second case, we would need $y=0$ and especially $y_n=0$ by the first line. This implies $b=0$ by the last line. But then, also all $c_i$ become zero and $z_i=y_i=0$ by the second line. So the zero vector should be in $S_a^{n-1}$ which is false as $a\notin S$. 
    Therefore, $S'=\bigcap_{a\in \IR^n\setminus S} S_a\subset S$.
    
    On the other hand take any $s\in S$ and $a \in \IR^n\setminus S$. The inequality $\langle x-a, v_a \rangle\ge 0$ is clear. The inequality $\langle x-a, v_a \rangle < R_a$ follows because we can write $x=b'+v$ with $b'\in B$, $v\in V$ and get 
    $$
        \langle x-a, v_a \rangle=\langle b'-a, v_a \rangle+\langle v, v_a \rangle \le \|b'-a\|\cdot \|v_a\|+0=\|b'-a\|< R_a
    $$
    with the Cauchy-Schwartz inequality and Lemma \ref{LemmaInnerVectorOrthogonal}. So in the first case of defining $S_a$, $s\in S_a$ is clear.
    
    For the second case, we need to distinguish again between $\langle x-a, v_a \rangle=0$ and $\langle x-a, v_a \rangle>0$.
    In the equal-zero-case, we can choose $x=s$ and $y=O_a(x-a)$. This implies $y_n=0$. 
    Choose furthermore $b=0,c=0$ and $z$ as the restriction of $y$. Then, $z\in S_a^{n-1}$, so every condition is satisfied.
    In the positive-case, let $d$ be a point in $B$ such that $\langle d-a, v_a \rangle=0$. 
    This point exists, because there is such a point in $S$ with $\langle d-a, v_a \rangle=0$ and subtracting a vector in $V$ does not change the scalar product by Lemma \ref{LemmaInnerVectorOrthogonal}. Choose $z$ as $O_a(d-a)$ restricted to the first $n-1$ entries and $y$ as $O_a(x-a)$. Now, $y_n$ is positive and not larger than $|y_n|\le \|y\|<R_a$ as $O_a$ is an isometry. Choose $b\coloneqq y_n/R_a$ and $c_i$ as needed. Note that the $c_i$ are freely enough, since $b>0$ and $|y_i-z_i|\le\|y\|+\|z\|<2R_a$. Therefore also in this case, we get $s\in S_a$.
    
    Finally, we have $S'=\bigcap S_a = S$, so the above is an infinitary pp-definition of $S$.
\end{proof}

\subsection{Definitions from the pointed stripe}
The pointed stripe can primitively positively define all convex sets without a ray intersection. This will be shown in this section.

\begin{thm} \label{TheoremDefiningFromStripe}
    Let $S\subset \IR^n$ be convex such that for each affine set $A$, the closure of the intersection of $A\cap S$ is the sum of a vector space and a compact set.
    Then, $S$ is infinitary pp-definable over $(\IR;+,\cdot \lambda \mid \lambda \in \IR, \lambda \mid \lambda \in \IR, \stripe)$. Here, $\stripe$ denotes  the set $\{(x,y)\in \IR^2\mid y \in (0,1) \lor (x,y)=(0,0)\}$.
\end{thm}
\begin{proof}
    We show the theorem by induction on $n$.
    
    If $n=0$, then $\IR^n$ is a single point. Both this point and not this point are definable.
    
    For $n>0$, let $a\in \IR^n\setminus S$. We may assume that $S$ is nonempty.
    By Theorem \ref{Theorem Hahn-Banach}, there is $v_a$ of norm 1 such that $\langle s-a, v_a\rangle\ge 0$ for all $s\in S$.
    When we look at the closure of $S$, this will be the sum of a bounded set $B$ and a vector space $V$. By Lemma \ref{LemmaAltDescriptionInnerVspace}, $V$ is the inner vector space and by and Lemma \ref{LemmaInnerVectorOrthogonal}, it only contains vectors which are orthogonal to $v_a$. Therefore, the set 
    $$
        \{x \in \IR \mid \exists s \in S, \langle s,v_a \rangle =x\}=\{x \in \IR \mid \exists b \in B, \langle b,v_a \rangle =x\}
    $$
    is bounded. Therefore there is $R_a>0$ such that this set is contained in $[0,R_a)$.
    Distinguish two cases:
    
    If there is no element $s$ in $S$ with $\langle s-a, v_a\rangle=0$, define $S_a$ as
    $$
        \left\{ x \in \IR^n \mitte \stripe (1,\frac{1}{R}\langle x-a, v_a\rangle)\right\}.
    $$
    
    If there is an element $s$ in $S$ with $\langle s-a, v_a\rangle=0$, choose an orthogonal matrix $O_a$ such that $O_av_a$ is the $n$-th basis vector $e_n$. Now, the set 
    $$
        S_a^{n-1}\coloneqq \{x \in \IR^{n-1} \mid \exists s \in S : O_a(s-a)=(x_1,\dots ,x_{n-1},0)^T\} \subset \IR^{n-1}
    $$
    is up to an affine map the intersection of $S$ and an affine set. Therefore, we can apply the induction hypothesis on this set.
    So $S_a^{n-1}$ is infinitary pp-definable. 
    Since $O_a$ preserves lengths, the transformed set $\{O_a(s-a)\mid s\in S\}$ is bounded in  the $n$-th component by $R_a$.
    
    Define $S_a$ as
    \begin{align*}
        \{x \in \IR^n \mid \exists y \in \IR^n, b\in \IR,c,z \in \IR^{n-1} &: 
        O_a(x-a)=y \land z \in S_a^{n-1} 
        \\&\land \bigwedge_{i=1}^{n-1} \stripe(c_i,b) 
        \land \bigwedge_{i=1}^{n-1} z_i=y_i+c_i 
        \\&\land y_n=R_a b\}
    \end{align*}
    in this case.
    
    Finally, define $S'$ as
    $$
        \left\{ x \in \IR^n \mitte \bigwedge_{a \in \IR^n \setminus S} S_a(x)\right\}.
    $$
    
    Now, look at any $a\in \IR^n\setminus S$. This is not in $S_a$ which is clear in the first case and in the second case, we would need $y=0$ and especially $y_n=0$. This implies $b=0$. But then, also all $c_i$ become zero and $z_i=y_i=0$. So the zero vector should be in $S_a^{n-1}$ which is false as $a\notin S$. Therefore, $S'=\bigcap S_a\subset S$.
    
    On the other side take any $s\in S$ and $a \in \IR^n\setminus S$. 
    The inequality $\langle x-a, v_a \rangle\ge 0$ is clear. The inequality $\langle x-a, v_a \rangle < R_a$ follows as discussed before. So in the first case of defining $S_a$, $s\in S_a$ is clear.
    
    In the second case, we need to distinguish again if $\langle x-a, v_a \rangle=0$ or if it is positive.
    In the equal-zero-case, we can choose $x=s$ and $y=O_a(x-a)$. This implies $y_n=0$. Choose furthermore $b=0,c=0$ and $z$ as the restriction of $y$. Then, $z\in S_a^{n-1}$, so every condition is satisfied.
    In the positive-case, let $d$ be a point in $S$ such that $\langle d-a, v_a \rangle=0$. This point exists by assumption. Choose $z$ as $O_a(d-a)$ restricted to the first $n-1$ entries and $y$ as $O_a(x-a)$. Now, $y_n$ is positive and bounded by $|y_n|=\langle d-a, v_a \rangle<R_a$. Choose $b\coloneqq y_n/R_a$ and $c_i=z_i-y_i$ as needed. Note that the $c_i$ are free, since $b>0$ and $b<1$. Therefore also in this case, we get $s\in S_a$.
    
    As a conclusion we have $S'=\bigcap S_a = S$, so the above is an infinitary pp-definition of $S$.
\end{proof}

%% file: 05Non-Implications.tex
\section{Non-definability}
\label{SectionNonImplications}
This section focuses on proving that it is sometimes impossible to infinitarily pp-define a given convex relation from other given convex relations. Some of these result would also follow from the results in Section \ref{SubsectionApplicationPolymorphisms}. However, we want to show the simpler proofs first.

\begin{lemma} \label{LemmaConvexPolyStable}
    Let $\tau$ be a family of convex sets. Let $S\subset \IR^n$ be non-convex. Then, $S$ is not polymorphism invariant under $(\IR;\tau)$ and thus cannot be infinitary pp-defined over $(\IR;\tau)$.
\end{lemma}
\begin{proof}
    A set $S$ is convex, if and only if for all $\lambda \in [0,1]$, the map 
    \begin{align*}
        P_\lambda \colon \IR \times \IR &\to \IR \\
        (x,y)&\mapsto\lambda x + (1-\lambda)y
    \end{align*}
    preserves $S$ respectively is a polymorphism of $(\IR;S)$. So if $S$ is non-convex, there is $\lambda$ in $[0,1]$ such that $P_\lambda$ is no polymorphism of $S$ but of all convex sets in $\tau$. Thus, $S$ cannot be infinitary pp-definable by Lemma \ref{LemmaDefinablePolyStable}.
\end{proof}
\begin{lemma} \label{LemmaAffinePolyStable}
    Let $\tau$ be a family of affine sets. Let $S\subset \IR^n$ be non-affine. Then, $S$ is not polymorphism invariant under $(\IR;\tau)$ and thus cannot be infinitary pp-defined over $(\IR;\tau)$.
\end{lemma}
\begin{proof}
    A set $S$ is affine, if and only if for all $\lambda \in \IR$, the map 
    \begin{align*}
        P_\lambda \colon \IR \times \IR &\to \IR \\
        (x,y)&\mapsto\lambda x + (1-\lambda)y
    \end{align*}
    preserves $S$ respectively is a polymorphism of $(\IR;S)$. So if $S$ is non-affine, then there is $\lambda$ in $\IR$ such that $P_\lambda$ is no polymorphism of $S$ but for all affine sets in $\tau$. Thus, $S$ cannot be infinitary pp-definable by Lemma \ref{LemmaDefinablePolyStable}.
\end{proof}
\begin{lemma} \label{LemmaIPPGeometrisch}
    Let $\tau$ be a family of subsets of $\IR^n$ for multiple $n$ such that the following holds:
    \begin{itemize}
        \item It contains the empty set and $\IR^1$.
        \item For each $S\subset \IR^n, S\in \tau$ and each affine map $f\colon \IR^n \to \IR^m$, the set $f(S)$ is in $\tau$.
        \item For each $T\subset \IR^m, T\in \tau$ and each affine map $f\colon \IR^n \to \IR^m$, the set $f^{-1}(T)$ is in $\tau$.
        \item For each $n$ and each subset $I \subset \tau\cap {\powerset(\IR^n)}$, the intersection $\bigcap_{S\in I} S\subset \IR^n$ is in $\tau$.
    \end{itemize}
    Then, $\tau$ is closed with respect to infinitary pp-definitions.
\end{lemma}
\begin{proof}
    Assume that $S$ is infinitary pp-definable by relations corresponding to sets in $\tau$. By induction on the formula, we may assume that $S$ is either an atomic formula or given by a single step, that is by Definition \ref{DefinitionIPPDefinable} either a conjunction or an existential quantifier.

    If $S$ is given by an atomic formula, it is directly given by a single relation $T\subset \IR^m$. Here, we have 
    $$(x_1,\dots,x_n)\in S \iff (x_{j_1},\dots ,x_{j_m}) \in T
    $$
    for some function $j\colon \{1,2,\dots,m\} \to \{1,2,\dots,n\}$. This transition of variables can be understood as a linear map $f\colon \IR^n \to \IR^m, (x_1,\dots,x_n) \mapsto (x_{j_1},\dots ,x_{j_m})$ in which sense $S=f^{-1}(T)$ is the inverse image of an affine map.

    If $S$ is given by a conjunction, we have 
    $$(x_1,\dots,x_n)\in S \iff \bigwedge_{i\in I} (x_1,\dots,x_n)\in S_i
    $$
    for some $S_i \in \tau$. Here, $S$ is clearly given by the intersection $\bigcap_{i\in I} S_i$ where we identify subsets of $\IR^n$ with their respective relations. 

    Assume that $S$ is given by an existential quantifier, that is 
    $$
        (x_1,\dots,x_n)\in S \iff \exists y : (x_1,\dots,x_n,y)\in T
    $$
    for a relation $T\in \tau$. Then, define $f\colon \IR^{n+1}\to \IR^n$ to be the linear map which omits the last coordinate. This results in $S=f(T)=\{f(t)\mid t \in T\}$, so $S$ is the image of $T$ under an affine map.

    In every case, $S$ is again in $\tau$, so $\tau$ is closed with respect to infinitary pp-definitions.
\end{proof}

\begin{lemma} \label{LemmaRayIntersectionStable}
    Let $\tau$ be a family of convex sets.
    Let $S$ be a convex set which is infinitary pp-definable in $(\IR;\tau)$ and has a ray intersection. Then, already a set in $\tau$ has a ray intersection.
\end{lemma}
\begin{proof}
    Let $\tau'$ be the set of all convex sets without a ray intersection. We show that $\tau'$ is closed with respect to infinitary pp-definitions. We do this with the help of Lemma \ref{LemmaIPPGeometrisch}.

    Clearly, $\tau'$ is nonempty. Moreover, all sets obtained from $\tau'$ by applying intersections, affine maps or inverse affine maps are again convex. So what left is, is to show that if a set $S$ with a ray intersection can be obtained from another set $T$ by an affine map, an inverse affine maps or from other sets intersections, then one of the sets which were used to obtain $S$ already had a ray intersection.
    %
    %By induction, we may assume that $S$ is directly defined by Elements in $\tau$ in one step. 
    Let $\lambda \mapsto \lambda w+v$ be the ray intersection of $S$.
    %If the formula for $S$ is atomic, $S$ is in $\tau$ and has a ray intersection.
    
    If $S$ is an intersection, we can write $S=\bigcap_{i \in I} S_i$. Then, there is $i\in I$ such that $(-1)w+v\notin S_i$, but of course $\lambda w+v \in S_i$ for all $\lambda>0$. So there is $v'$ in the line segment between $v$ and $v-w$ such that $\lambda \mapsto v' + \lambda w$ is a ray intersection for $S_i$.

    If $S$ is given by an inverse affine map or invertible affine map, we can write $S=\{x \in \IR^n \mid f(x)\in T\}$. In this case,  $\lambda \mapsto f(\lambda w+v)$ is a ray intersection for $T$. This works similar for injective affine maps.

    An affine map can be considered as the composition of injective affine maps and projections.
    If $S$ is given by a projection, take $T$ as the inverse image such that
    $$
        S=\{x \in \IR^n \mid \exists y\in \IR: (x,y)\in T\}.
    $$
    Let $A$ be the affine space which is mapped onto $\{\lambda w+v \mid \lambda \in \IR\}$. Look at the closure $\closure{A\cap T}$. This closure cannot be the sum of a bounded set and a vector space, since it contains inverse images for arbitrarily large $\lambda$ but not for negative $\lambda$. By Lemma \ref{LemmaDichotomieClosed}, $\closure{A\cap T}$ has a ray intersection. By Corollary \ref{CorollaryMoveToInner}, also $A\cap T$ has a ray intersection. Since this is an intersection and $A$ has no ray intersection, we get that $T$ has a ray intersection.
    
    So in all cases, there was a set with a ray intersection before.
\end{proof}

\begin{lemma}\label{LemmaBoundedPerserved}
    Let $\tau$ be a set of relation Symbols where each symbol is given by the sum of a bounded set and a vector space. Let $S$ be infinitary pp-definable from $(\IR;\tau)$. Then, $S$ is given by the sum of a bounded set and a vector space.
\end{lemma}
\begin{proof}
    It suffices to show that the set of sets which are a sum of a bounded set and a vector space are closed under inverse affine maps, affine maps and intersections by Lemma \ref{LemmaIPPGeometrisch}.
    
    For affine maps and inverse affine maps, this is clear.
    
    %For affine maps this is clear: The projection of bounded sets are bounded, the projection of a vector space is a vector space and sums are preserved.

    %For inverse affine maps this is clear: The inverse image of a vector space is a vector space and the inverse image of a bounded set is a bounded set plus the kernel of this map, which is a vector space. Moreover, sums are preserved.
    
    For intersections write $S=\bigcap_{i\in I} S_i$ and $S_i=B_i+V_i$ where $V_i$ is a vector space and $B_i$ is bounded. 
    Let $V'=\bigcap_{i\in I} V_i$. 
    Since the directions in $V'$ is independent from the rest, we may assume that $V'=\{0\}$ by projecting in the quotient space and at the end taking the sum with $V'$ again.
    Now, $\bigcap_{i\in I} V_i=\{0\}$. 
    Since this is an intersection of finite dimensional vector spaces, there is a finite subset $J\subset I$ such that $\bigcap_{j\in J} V_j=\{0\}$. 
    By repeatedly applying Lemma \ref{LemmaIntersectionStaysBounded}, we get that $\bigcap_{j\in J} S_j$ is the sum of an vector space and a bounded set. Since the vector space needs to be zero, it is a bounded set. Therefore,
        $$
            \bigcap_{i\in I} S_i \subset \bigcap_{j \in J} S_j
        $$
    is a subset of a bounded set and thus itself bounded.
    
    Finally, we get in any case that $S$ is a sum of a bounded set and an affine space.
\end{proof}

\begin{lemma} \label{LemmaIndependenceCompact}
    Let $\tau$ be a set of relation symbols where each symbol is given by the sum of a compact set and a vector space. Let $S$ be infinitary pp-definable from $(\IR;\tau)$. Then, $S$ is given by the sum of a compact set and a vector space.
\end{lemma}
\begin{proof}
    It suffices to show that these sets are closed under affine maps, inverse affine maps and intersections by Lemma \ref{LemmaIPPGeometrisch}.
    
    For affine maps this follows because they are continuous so the image of a compact set is compact and by linearity, the image of a vector space is a vector space which implies that the image of the sum of them is again the sum of a vector space and a compact set.

    For inverse affine maps, recall that a subset of $\IR^n$ is the sum of a compact set and a vector space if and only if it is closed and the sum of a bounded set and a vector space. This follows directly from Lemma \ref{LemmaDichotomieClosed}. Since affine maps are continuous, the inverse image of closed sets are closed. Moreover, the inverse image of the sum of a bounded set and a vector space is again the sum of a bounded set and a vector space as proven in the proof of Lemma \ref{LemmaBoundedPerserved}.
    
    For intersections we get by Lemma \ref{LemmaBoundedPerserved} that the result is the sum of a vector space and a bounded set. Since intersections of closed sets are closed, it is moreover a closed set. Therefore, it is also the sum of a closed bounded set and a vector space.
\end{proof}

\begin{lemma} \label{LemmaClosurePreimageCommute}
    Let $T\subset \IR^n$ be a convex set which is given by the sum of a bounded set and a vector space. Let $p\colon \IR^{n+1} \to \IR^n$ be a projection. Then, $p(\closure{T})=\closure{p(T)}$.
\end{lemma}

\begin{proof}
    Let $S$ be the image of $T$. By applying an invertible affine map, we may assume that $p$ is the projection onto the first $n$ coordinates. Let $t$ be the $(n+1)$-th unit vector. Then, $S=\{s\in \IR^n \mid \exists \lambda \in \IR, s+\lambda t \in T\}$ and $\|t\|=1$. Write $T=B+V$ for a bounded set $B$ and a vector space $V$ and choose $a\in \closure{S}\setminus S$.
    
    If $t\in V$, then $T$ is the sum of $S$ and $\{\lambda t \mid \lambda \in \IR\}$ and every inverse image of $a$ is in $\closure{T}$.
    
    If $t\notin V$, then look at $T\cap (\ball_\epsilon(a)+\{\lambda t \mid \lambda \in \IR\})$ for some $\epsilon>0$. This set is the sum of a bounded set and a vector space by Lemma \ref{LemmaBoundedPerserved}. Since $t\notin V$ it is in fact bounded. Let $R$ be a bound for $\epsilon=1$ Since $a\in \closure{S}$, we furthermore know that this set is nonempty for all $\epsilon>0$. Let $(b_n)_{n\in \IN}$ be a sequence such that 
    $$
        b_n\in T\cap (\ball_{\frac{1}{n}}(a)+\{\lambda t \mid \lambda \in \IR\})
    $$
    holds for all $n\in \IN$. Since this sequence is bounded by $R$ and $b_n\in T$, it has an accumulation point $b \in \closure{T}$. Furthermore, $b$ will be in $\{a\}+\{\lambda t \mid \lambda \in \IR\}$, so it will be an inverse image of $a$ which is in $\closure{T}$. Thus, $p(\closure{T})\supset \closure{p(T)}$.
    
    The direction $p(\closure{T})\subset \closure{p(T)}$ follows as the projection is a continuous map.% which for example preserves Cauchy sequences.
\end{proof}

\begin{lemma} \label{LemmaNoSemidentIndependent}
    Let $\tau$ be a set of convex relations where each relation is given by the sum of a bounded set and a vector space and none of them has a \dent{}. 
    Let $S$ be infinitary pp-definable from $(\IR;\tau)$. Then, $S$ has the same properties.
\end{lemma}

\begin{proof}
    It suffices to show that the set of all sums of a bounded set and a vector space without a \dent{} is 
    closed under affine maps, inverse affine maps and intersections. For the property of being the sum of a bounded set and a vector space, this was shown in Lemma \ref{LemmaBoundedPerserved}.
    
    For intersections, let $S=\bigcap_{i \in I} S_i$. Then, assume that $a\notin S$ is a \dent{}. Then, there is $i$ with $a\notin S_i$. So there is $s\in S$ and $t\in \closure{S}$ and $\lambda \in (0,1)$ such that $a=\lambda s+(1-\lambda)t$. But since $s\in S_i$ and $t\in \closure{S_i}$, we get that $a$ was a \dent{} for $S_i$.

    We split the affine maps into injective affine maps and projections. The injective affine maps clearly preserve the non-existence of a \dent{}.
    For projections let $T\subset \IR^{n+1}$ be represented in $\tau$ and assume $S$ is a projection of $T$. Then, assume that $a\notin S$ is a \dent{} for $S$ and $a=\lambda s + (1-\lambda) t$ with $s\in S, t\in \closure{S}, \lambda \in (0,1)$. Let $s'\in T$ be an inverse image of $s$ and $t'\in \closure{T}$ an inverse image of $t$ which exists by Lemma \ref{LemmaClosurePreimageCommute}. Then, $\lambda s' + (1-\lambda)t'$ is contained in the inverse image of $a$ and thus a \dent{}.

    For inverse affine maps, let $S=\{s \mid f(s)\in T\}$. If $a=\lambda s+(1-\lambda)t$ is a \dent{} of $S$, then $f(a)=\lambda f(s) + (1-\lambda) f(t)$ is a \dent{} of $T$. Thus, if $T$ has no such point, the same holds for $S$.
    
    So in each case, the properties are preserved.
\end{proof}

%% file: 06Logical-Corollaries.tex
\section{Detailed results and corollaries}
\label{SectionCorollaries}
In addition to the short characterisations of the equivalence classes we gave in Section \ref{SectionResultsShort}, we get the following additional descriptions and corollaries of the results.

\subsection{Characterisations of the equivalence classes}
We recall equivalent conditions of having a ray intersection and prove Theorem \ref{TheoremDescriptionOfEqClasses}.

\begin{thm} \label{TheoremFall1Ausgerollt}
    For a convex subset $S\subset \IR^n$, the following are equivalent:
    \begin{enumerate} 
        \item \label{NrTheoremFall1Ausgerollt1}
        The sets $S$ and $\{x\in \IR \mid x\ge 0\}$ are infinitary pp-interdefinable over $(\IR, +, \cdot \lambda\mid \lambda\in \IR, \lambda\mid \lambda \in \IR)$.
        \item \label{NrTheoremFall1Ausgerollt4}
        The set $\{x\in \IR \mid x\ge 0\}$ is infinitary pp-definable over $(\IR, +, \cdot \lambda\mid \lambda\in \IR, \lambda\mid \lambda \in \IR, S)$.
        \item \label{NrTheoremFall1Ausgerollt3}
        The set $\{x\in \IR \mid x>0\}$ is infinitary pp-definable over $(\IR, +, \cdot \lambda\mid \lambda\in \IR, \lambda\mid \lambda \in \IR, S)$.
        \item \label{NrTheoremFall1Ausgerollt5}
        Every convex set is infinitary pp-definable over $(\IR, +, \cdot \lambda\mid \lambda\in \IR, \lambda\mid \lambda \in \IR, S)$.
        \item \label{NrTheoremFall1Ausgerollt6}
        The set $S$ is not infinitary pp-definable over $(\IR, +, \cdot \lambda\mid \lambda\in \IR, \lambda\mid \lambda \in \IR, \stripe)$ where $\stripe$ denoted the relation symbol corresponding to the set $\{(x,y) \in \IR^2 \mid y \in (0,1) \lor (x,y)=(0,0)\}$.
        \item \label{NrTheoremFall1Ausgerollt7}
        For each affine space $A$, the closure $\closure{S\cap A}$ is not the sum of a compact set and a vector space.
        \item \label{NrTheoremFall1Ausgerollt8}
        It has a ray intersection: There are vectors $v,w\in \IR^n$ such that $\{v+\lambda w \mid \lambda >0\} \subset S$ and $\{v+\lambda w \mid \lambda <0\} \cap S=\emptyset$.
        \item \label{NrTheoremFall1Ausgerollt2}
        The set $\{x\in \IR \mid x>0\}$ or the set $\{x\in \IR \mid x\ge 0\}$ is pp definable over $(\IR, +, \cdot \lambda\mid \lambda\in \IR, \lambda\mid \lambda \in \IR, S)$.
    \end{enumerate}
\end{thm}
\begin{proof}
    We prove the following implications:
    $$
    \begin{tikzcd}
        1 \ar[leftrightarrow, r]& 2 \ar[r] & 5 \ar[r] & 6 \ar[leftrightarrow, d] \\
        4 \ar[leftrightarrow, ru] & 3 \ar[leftrightarrow, u]  & 8 \ar[l] & 7 \ar[l]
    \end{tikzcd}
    $$

    \ref{NrTheoremFall1Ausgerollt1} implies \ref{NrTheoremFall1Ausgerollt4} is trivial,
    \ref{NrTheoremFall1Ausgerollt4} implies \ref{NrTheoremFall1Ausgerollt1} by Theorem \ref{TheoremDefiningFromRay}.
    
    \ref{NrTheoremFall1Ausgerollt4} and \ref{NrTheoremFall1Ausgerollt3} are equivalent by Lemma \ref{LemmaInterDefinableHalveIntervalls}.
    
    \ref{NrTheoremFall1Ausgerollt4} implies \ref{NrTheoremFall1Ausgerollt5} by Theorem \ref{TheoremDefiningFromRay}, \ref{NrTheoremFall1Ausgerollt5} implies \ref{NrTheoremFall1Ausgerollt4} is trivial. 
    
    \ref{NrTheoremFall1Ausgerollt4} implies \ref{NrTheoremFall1Ausgerollt6} by Lemma \ref{LemmaRayIntersectionStable} since the closed ray has a ray intersection and the pointed stripe has not.
    
    \ref{NrTheoremFall1Ausgerollt6} implies \ref{NrTheoremFall1Ausgerollt7} by Theorem \ref{TheoremDefiningFromStripe} and a contraposition.
    
    \ref{NrTheoremFall1Ausgerollt7} is equivalent to \ref{NrTheoremFall1Ausgerollt8} by Theorem \ref{TheoremClosureIffRay}.
    
    \ref{NrTheoremFall1Ausgerollt8} implies \ref{NrTheoremFall1Ausgerollt2}, \ref{NrTheoremFall1Ausgerollt2} implies \ref{NrTheoremFall1Ausgerollt3} and \ref{NrTheoremFall1Ausgerollt4} by Lemma \ref{LemmaImplicationsToRay}.
\end{proof}

\begin{thm}[Theorem \ref{TheoremDescriptionOfEqClasses}]
    A convex set $S\subset \IR^n$ is infinitary pp-interdefinable over $(\IR;+,\cdot c \mid c \in \IR, c\mid c \in \IR)$
    \begin{enumerate}
        \item \label{PartTheoremDescriptionOfEqClasses1}
        with $\{0\}\in \IR$ 
        if and only if it is empty or an affine subspace.
        \item \label{PartTheoremDescriptionOfEqClasses2}
        with $\{x \in \IR \mid x \in [0,1]\}$ 
        if and only if it is not an affine subspace but the sum $\{b+v\mid b \in B, v \in V\}$ of a compact convex set $B\subset \IR^n$ and a vector subspace $V\subset \IR^n$.
        \item \label{PartTheoremDescriptionOfEqClasses3}
        with $\{x \in \IR \mid x \in (0,1)\}$ 
        %if and only if it is the sum $\{b+v\mid b \in B, v \in V\}$ of a bounded, convex set $B\subset \IR^n$ and a vector space $V\subset \IR^n$, $B$ cannot be chosen to be compact and there is no \dent{} $a\notin S$.
        if and only if it is not closed, there is no \dent{} $a\notin S$, and $S$ is the sum $\{b+v\mid b \in B, v \in V\}$ of a bounded, convex set $B\subset \IR^n$ and a vector subspace $V\subset \IR^n$.
        \item \label{PartTheoremDescriptionOfEqClasses4}
        with $\{(x,y) \in \IR^2 \mid (x\in (-1,1)\land y \in (0,1)) \lor (x,y)=(0,0)\}$ 
        if and only if it is the sum $\{b+v\mid b \in B, v \in V\}$ of a bounded, convex set $B\subset \IR^n$ and a vector subspace $V\subset \IR^n$ and there is a \dent{} $a\notin S$.
        \item \label{PartTheoremDescriptionOfEqClasses5}
        with $\{(x,y) \in \IR^2 \mid y \in (0,1) \lor (x,y)=(0,0)\}$ 
        if and only if it is not the sum of a vector space and a bounded set, and does not contain a ray intersection. 
        \item \label{PartTheoremDescriptionOfEqClasses6}
        with $\{x \in \IR \mid x\ge 0\}$ 
        if and only if it contains a ray intersection.
    \end{enumerate}
    Every convex set is in exactly one of the above classes. Moreover, the sets which are infinitary pp-definable over the structure $(\IR;+,\cdot c \mid c \in \IR, c\mid c \in \IR, S)$ are exactly the sets in the classes with a lower or equal number than $S$. 
    No non-convex set is infinitary pp-definable from convex sets.
\end{thm}

\begin{proof} %\todo{Tabellarisch}
    We start with showing the equivalences. For a better readebillity, we shorten ``infinitary pp-definability" to ``definability".
    
    Equivalence \ref{PartTheoremDescriptionOfEqClasses1}: Since $\{0\}$ is definable, being interdefinable with this set is the same as being definable over $(\IR;+,\cdot c \mid c \in \IR, c\mid c \in \IR)$. Every affine space is definable over $(\IR;+,\cdot c \mid c \in \IR, c\mid c \in \IR)$ by Theorem \ref{TheoremAffineFormNothing}. Conversely, every set which is definable over $(\IR;+,\cdot c \mid c \in \IR, c\mid c \in \IR)$ is affine by Lemma \ref{LemmaAffinePolyStable}.
    
    Equivalence \ref{PartTheoremDescriptionOfEqClasses2}: If $S$ is interdefinable, then it is definable by sets which are given as sum of compact sets and vector spaces. Thus it is itself of this shape by Lemma \ref{LemmaIndependenceCompact}. It cannot be affine, because affine sets cannot define non affine sety by Lemma \ref{LemmaAffinePolyStable}. 
    Conversely if $S$ has the described shape, it is definable from $[0,1]$ by Theorem \ref{TheoremCompactVectDefinable} and it defines $[0,1]$ by Lemma \ref{LemmaDefiningCompactIntervall}.
    
    Equivalence \ref{PartTheoremDescriptionOfEqClasses3}: If $S$ is interdefinable, it is definable and thus is given by the sum of a vector space and a bounded set without a \dent{} by Lemma \ref{LemmaNoSemidentIndependent}. Since it also defines $(0,1)$, it cannot be the sum of a compact set and a vector space by Lemma \ref{LemmaIndependenceCompact} and thus is not closed by Lemma \ref{LemmaClosedEqualsCompact}.
    Conversely, if $S$ has the describes shape, it is definable from $(0,1)$ by Theorem \ref{TheoremDefinitionFromOpenIntervall} and it cannot be a closed set and thus defines $(0,1)$ by Lemma \ref{LemmaDefiningOpenIntervall}.
    
    Equivalence \ref{PartTheoremDescriptionOfEqClasses4}: If $S$ is interdefinable, it is definable and thus is given by the sum of a vector space and a bounded set by Lemma \ref{LemmaBoundedPerserved}. Since it also defines the set with a \dent{}, it needs to have itself a \dent{} and cannot be the sum of a compact set and a vector space by Lemma \ref{LemmaNoSemidentIndependent}. 
    Conversely, if $S$ has the describes shape, it is definable from the pointed rectangle by Theorem \ref{TheoremBoundedDefining} and it defines the pointed rectangle by Lemma \ref{LemmaDefiningPointedRectangle}.
    
    Equivalence \ref{PartTheoremDescriptionOfEqClasses5}:
    By Theorem \ref{TheoremFall1Ausgerollt}, equivalence of \ref{NrTheoremFall1Ausgerollt6} and \ref{NrTheoremFall1Ausgerollt8}, $S$ has no ray intersection if and only if $S$ defines the pointed stripe.
    If $S$ is inter definable it thus has no ray intersection. Since $S$ also defines the pointed stripe which is not the sum of a bounded set and a vector space, $S$ is itself not the sum of a bounded set and a vector space by Lemma \ref{LemmaBoundedPerserved}.
    Conversely, if $S$ has the describes shape, it is definable from the pointed stripe. Moreover, it defines the pointed stripe by Lemma \ref{LemmaDefiningPointedStripeLimited}.
    
    Equivalence \ref{PartTheoremDescriptionOfEqClasses6}:
    This was already shown in Theorem \ref{TheoremFall1Ausgerollt} as equivalence of \ref{NrTheoremFall1Ausgerollt1} and \ref{NrTheoremFall1Ausgerollt8}.
    
    That each convex set is contained in exactly one of these cases follows from the geometric descriptions. To see this, note that having a ray intersection implies that the set is not the sum of a bounded set and a vector space and that the sum of a compact convex set and a vector space is closed by Lemma \ref{LemmaClosedEqualsCompact} and in particular cannot have a \dent{}.
    
    No non-convex set is infinitary pp-definable from the convex sets by Lemma \ref{LemmaConvexPolyStable}. No convex set is infinitary pp-definable from a convex set in a lower numbered equivalence class by Lemmas \ref{LemmaAffinePolyStable}, %1-2
    \ref{LemmaNoSemidentIndependent}, %2-3
    \ref{LemmaIndependenceCompact}, %3-4
    \ref{LemmaBoundedPerserved} and  %4-5
    \ref{LemmaRayIntersectionStable}. %5-6
    
    Conversely, every set infinitary pp-defines the set which describes its equivalence class and thus every element of a lower equivalence class by Theorems \ref{TheoremAffineFormNothing}, \ref{TheoremCompactVectDefinable}, \ref{TheoremDefinitionFromOpenIntervall}, \ref{TheoremBoundedDefining}, \ref{TheoremDefiningFromStripe} and \ref{TheoremDefiningFromRay}.
\end{proof}

\subsection{Geometric descriptions}
The property of infinitary primitively positively define other sets divides the convex sets into 6 classes as shown in the previous section. We now want to link this structure theoretic notion of definability with an elementary geometric description and prove the equivalence of the first three characterisations in Theorem \ref{TheoremMainGeometrisch}. %\todo{Einleitung korrekturlesen}

\begin{thm}[First part of Theorem \ref{TheoremMainGeometrisch}]
\label{TheoremIPPGeometrisch2}
    Consider the structure $(\IR;+,\cdot c \mid c \in \IR, c\mid c \in \IR)$ and a set $\tau \subset \bigsqcup_{n\in \IN} \powerset(\IR^n)$ of some subsets of real vector spaces. Assume that $\tau$ contains only convex sets. Then, the following are equivalent:
    \begin{enumerate}
        \item \label{PartTheoremMainGeometrisch1}
        The set $\tau$ is closed under infinitary pp-definitions over the structure $(\IR;+,\cdot c \mid c \in \IR, c\mid c \in \IR)$.
        \item \label{PartTheoremMainGeometrisch2}
        The set $\tau$ has the following properties:
        \begin{itemize}
            \item It contains the empty set and $\IR^1$.
            \item For each $S\subset \IR^n\in \tau$ and each affine map $f:\IR^n \to \IR^m$, the set $f(S)$ is in $\tau$.
            \item For each $T\subset \IR^m\in \tau$ and each affine map $f:\IR^n \to \IR^m$, the set $f^{-1}(T)$ is in $\tau$.
            \item For each $n$ and each subset $I \subset \tau\cap {\powerset(\IR^n)}$, the intersection $\bigcap_{S\in I} S\subset \IR^n$ is in $\tau$.
        \end{itemize}
        \item \label{PartTheoremMainGeometrisch3}
        The set $\tau$ is one of the following 6 sets:
        \begin{itemize}
            \item All affine sets,
            \item All sums $\{b+v\mid b\in B, v\in V\}$ of a vector space $V$ and a compact convex set $B$,
            \item All sums $\{b+v\mid b\in B, v\in V\}$ of a vector space $V$ and a bounded convex set $B$ where the sum has no \dent{},
            \item All sums $\{b+v\mid b\in B, v\in V\}$ of a vector space $V$ and a bounded convex set $B$,
            \item All convex sets without a ray intersection or
            \item All convex sets.
        \end{itemize}
    \end{enumerate}
    The equivalence of \ref{PartTheoremMainGeometrisch1} and \ref{PartTheoremMainGeometrisch2} also holds for non convex sets.
\end{thm}
\begin{proof}
    \ref{PartTheoremMainGeometrisch2} implies \ref{PartTheoremMainGeometrisch1} by Lemma \ref{LemmaIPPGeometrisch}.
    
    \ref{PartTheoremMainGeometrisch1} implies \ref{PartTheoremMainGeometrisch2}: If $\tau$ is closed under infinitary pp-definitions then it contains every affine set by Theorem \ref{TheoremAffineFormNothing}. Moreover, if $S$ is infinitary pp-definable, then $f(S)$ is definable as
    $$
        \{y \in \IR^m \mid \exists x \in \IR^n : f(x)=y \land x\in S\}
    $$
    where $f(y)=x$ is a system of linear equations. Similar $f^{-1}(T)$ is definable as
    $$
        \{x \in \IR^n \mid \exists y \in \IR^m : f(x)=y \land y\in T\}
    $$
    where $f(y)=x$ is again a system of linear equations. Finally, the infinite intersection is definable as infinite conjunction.

    \ref{PartTheoremMainGeometrisch1} is equivalent to \ref{PartTheoremMainGeometrisch3}: This is a direct consequence of Theorem \ref{TheoremDescriptionOfEqClasses}.
\end{proof}

The equivalent characterisation with linear maps will be proven in Theorem \ref{TheoremIPPGeometrisch3}.

\subsection{Application to polymorphisms}
\label{SubsectionApplicationPolymorphisms}
Until now, we have considered convex sets and infinitary primitive positive definability. We have only slightly touched the world of polymorphisms. Hoever, they are closely related:
If $f:\setstructA^\alpha \to \setstructA$ is a (possibly infinite arity) polymorphism of any $\sigma$-structure $\structA$ and $S$ is a primitively positively definable relation on $\setstructA$, then $S$ is also invariant under the the polymorphism. Thus we get a chain by looking at all relations on $\setstructA$ that are
\begin{enumerate}
    \item infinitarily primitively positively definable,
    \item preserved by all (possibly infinite arity) polymorphisms,
    \item preserved by all polymorphisms on $\setstructA$ of at most countable arity,
    \item preserved by all polymorphisms on $\setstructA$ of finite arity
\end{enumerate}
where each set is contained in the next one. It is worth to mention that in many possible definitions, polymorphims and their invariant relations form a Galois connection. For a systematic overview to this topic, see \cite[Table 1]{PoePotsd01}.

It turns out that in our case, that is if $\sigma$ contains $(+,\cdot c\mid c\in \IR, c\mid c\in \IR)$ and is limited to convex sets, the first three sets are the same and given by the set of the six families described in Theorem \ref{TheoremMainGeometrisch}. To prove this, consider the following lemma:
\begin{lemma} \label{LemmaTrennerPolys}
    Let $S_k$ with $k\in \{1,2,\dots,6\}$ be the six sets representatives defined in Theorem \ref{TheoremDescriptionOfEqClasses}. Chose any such $k$. Then, there is a set of polymorphisms $f\colon \IR^\IN \to \IR$ of arity $\omega$ which preserve $(\IR;+,\cdot c \mid c \in \IR, c\mid c \in \IR, S_k)$ but not $S_{k+1}$ if $k<6$ and no nonconvex sets. 
\end{lemma}%\todo{Lemma + Beweis korrekturlesen}
\begin{proof}
    We consider 0 to be a natural number.
    Let $V\subset \IR^\IN$ be the linear subspace of bounded sequences. Let $\filt$ be a free ultrafilter on $\IN$. Define $g_i\colon V \to \IR$ and $y_i,z_i\in \IR^\IN$ for $i\in \{1,2,3,4,5\}$ as follows:
    \begin{align*}
        g_1((x_n)_{n\in \IN})&\coloneqq x_0-x_1+x_2
        &
        y_{1,n}&\coloneqq z_{1,n}\coloneqq \begin{cases} 0 \text{ if $n$ is even} \\ 1 \text{ else}\end{cases}
        \\
        g_2((x_n)_{n\in \IN})&\coloneqq \lim_{n\in \filt} x_n
        &
        y_{2,n}&\coloneqq z_{2,n}\coloneqq n^{-1}
        \\
        g_3((x_n)_{n\in \IN})&\coloneqq \frac{1}{2}\lim_{n\in \filt} x_n + \frac{1}{2} x_0
        &
        y_{3,n}&\coloneqq \begin{cases} 0 \text{ if $n=0$} \\ \frac{1}{2} \text{ else}\end{cases}
        \\ &&
        z_{3,n}&\coloneqq \begin{cases} 0 \text{ if $n=0$} \\ \frac{1}{n} \text{ else}\end{cases}
        \\
        g_4((x_n)_{n\in \IN})&\coloneqq x_0
        &
        y_{4,n}&\coloneqq  \begin{cases} 0 \text{ if $n=0$} \\ \frac{1}{2} \text{ else}\end{cases}
        \\ &&
        z_{4,n}&\coloneqq  n
        \\
        g_5((x_n)_{n\in \IN})&\coloneqq \sum_{n=0}^\infty 2^{-(n+1)}x_n
        &
        y_{5,n}&\coloneqq z_{5,n}\coloneqq n 
    \end{align*}
    Note that $g_2$ is well-defined, because the sequence is bounded and every bounded ultrafilter converges.
    Every $g_i$ is linear and $g_i$ applied to a constant function is the constant. Let $(f_i)_{i\in \{1,2,3,4,5\}}\colon \IR^n \to \IR$ be any linear map such that $f_i|_V=g_i$ and $f_i(0,1,2,\dots)=f_i((n)_{n\in \IN})=-1$.

    Now, $f_i$ are maps $\IR^\IN \to \IR$ which preserve $(+,\cdot c \mid c \in \IR)$ as they are vector spaces and they preserve $(c\mid c \in \IR)$, because the constant (and thus bounded) sequence $(c,c,\dots)$ is mapped to $c$ for all $c\in \IR$. Thus, all of those maps preserve all affine sets. 

    Furthermore, $g_2$ preserves the compact interval $[0,1]$, because every filter in a compact set converges in that set. The map $g_3$ preserves $(0,1)$, because $\lim_{n\in \filt} x_n$ will be in the closure $[0,1]$ and $x_0$ will be inside the interior $(0,1)$. 
    Thus, their average is inside the open interval. The map $g_4$ preserves all bounded sets including the realisation of $\rect$, because if $B$ is a bounded set, then every sequence in $B$ is also bounded and thus will be projected onto their first component $x_0$ which is inside $B$. Finally, $g_5$ preserves the pointed stripe $\stripe$ by the following argument: 
    Let $(x_n,x_n')_{n\in \IN}$ be a sequence in $\stripe\subset \IR^2$. Then, all of the $x_n$ are in $[0,1)$. If at least one of them is in $(0,1)$, then the sum $\sum_{n=0}^\infty 2^{-(n+1)}x_n$ is also positive and thus the image $(\sum_{n=0}^\infty 2^{-(n+1)}x_n,\sum_{n=0}^\infty 2^{-(n+1)}x_n')$ in $\stripe$. In the other case, where all $x_n$ are zero, we get since $(x_n,x_n')\in \stripe$ that also all $x_n'$ are zero. In this case, the image under $f_5$ is $(0,0)$, which is inside $\stripe$.

    We also get some non-preservation: The map $f_1$ does not preserve $[0,1]$, because $(y_{1,n})_{n\in \IN}$ is in $[0,1]^n$, but $f_1((x_n)_{n\in \IN})=-1$ is not. The image $f_2((y_{2,n})_{n\in \IN})$ is zero, because the limit of the sequence is zero, thus every free ultrafilter will converge to zero as well. This shows that $f_2$ does not preserve $(0,1)$. The map $f_3$ does not preserve $\rect$, because for every $n$, we have $(y_{3,n},z_{3,n})\in \rect$, but $f_3((y_{3,n},z_{3,n})_{n \in \IN})=(\frac{1}{4}, 0)\notin \rect$. We get similar that $(y_{4,n},z_n)\in \stripe$, but $f_4((y_{4,n},z_n)_{n\in \IN})=(0,-1)\notin \stripe$ thus $f_4$ does not preserve $\stripe$. Finally, $f_5((y_{5,n})_{n \in \IN})=-1$ thus $f_5$ does nor preserve the set of non-negative real numbers.

    To come back to the original question, we see that the set $\{(x_n)_{n\in \IN}\mapsto \lambda x_0 + (1-\lambda)x_1 \mid \lambda \in [0,1]\}$ preserves all convex sets and no other set by the definition of a convex set. This solves $k=6$. For $k<6$, take instead the set $\{f_k\}\cup \{(x_n)_{n\in \IN}\mapsto \lambda x_0 + (1-\lambda)x_1 \mid \lambda \in [0,1]\}$. We showed that it has the desired properties of preserving $S_k$ but not $S_{k+1}$.
\end{proof}
\begin{cor} \label{CorollaryCountablePolysTrennen}
    We consider the structure $(\IR;+,\cdot c \mid c \in \IR, c\mid c \in \IR, S)$ where $S$ is any convex set. Then, the infinitary pp-definable sets are exactly the sets which are invariant under all countable arity polymorphisms.
\end{cor}
\begin{proof}
    Every set which is infinitary pp-definable is also invarant under all polymorphisms by Lemma \ref{LemmaDefinablePolyStable}. 
    Therefore, we also get that the polymorphisms of $(\IR;+,\cdot c \mid c \in \IR, c\mid c \in \IR, S)$ are exactly the polymorphisms of $(\IR;+,\cdot c \mid c \in \IR, c\mid c \in \IR, S_k)$ where $S_k$ is the set which is infinitary pp-interdefinable with $S$ and a representative of this equivalence class in Theorem \ref{TheoremDescriptionOfEqClasses}.
    Take any set $T$ which is not infinitary pp-definable. 
    By Theorem \ref{TheoremDescriptionOfEqClasses} we get that $T$ is nonconvex or $k<6$ and $T$ infinitary pp-defines $S_{k+1}$. 
    If every polymorphism of $(\IR;+,\cdot c \mid c \in \IR, c\mid c \in \IR, S_k)$ would preserve $T$, it would also preserve sets that are infinitary pp-definable from $T$. 
    In both cases, we get by Lemma \ref{LemmaTrennerPolys} that $T$ is not preserved under a countable arity polymorphism.
\end{proof}
This allows us to prove the rest of Theorem \ref{TheoremMainGeometrisch}:

\begin{thm}[Second part of Theorem \ref{TheoremMainGeometrisch}] \label{TheoremIPPGeometrisch3}
    Consider the structure $(\IR;+,\cdot c \mid c \in \IR, c\mid c \in \IR)$ and a set $\tau \subset \bigsqcup_{n\in \IN} \powerset(\IR^n)$ of some subsets of real vector spaces. Assume that $\tau$ contains only convex sets. Then, the following are equivalent:
    \begin{enumerate}
        \item \label{PartTheoremIPPGeometrisch3-1}
        The set $\tau$ is closed under infinitary pp-definitions over the structure $(\IR;+,\cdot c \mid c \in \IR, c\mid c \in \IR)$.
        \addtocounter{enumi}{2}
        \item \label{PartTheoremIPPGeometrisch3-4}
        There is a subset of the set $\{ f\colon \IR^\IN \to \IR\mid f((c)_{n\in \IN})=c \}$ of linear maps such that the set $\tau$ is the family of all convex sets that are preserved by all of those maps.
    \end{enumerate}
\end{thm}
\begin{proof}
    \ref{PartTheoremIPPGeometrisch3-4} implies \ref{PartTheoremIPPGeometrisch3-1}: By Lemma \ref{LemmaDefinablePolyStable}, the set $\tau$ is closed under infinitary pp-definitions. It contains addition and scalars, because all of those maps are linear and it contains constants by assumption.

    \ref{PartTheoremIPPGeometrisch3-1} implies \ref{PartTheoremIPPGeometrisch3-4}: Every set which is closed under infinitary pp-definitions over this structure, and only contains convex relations, is given by the infinitary pp-definable sets over $(\IR;+,\cdot c \mid c \in \IR, c\mid c \in \IR, S)$, where $S$ is a single convex set, by Theorem \ref{TheoremDescriptionOfEqClasses}. Now, Corollary \ref{CorollaryCountablePolysTrennen} shows that there are polymorphisms $\IR^\IN \to \IR$ which only preserve these relations. As these polymorphisms preserve $(+,\cdot c \mid c \in \IR, c\mid c \in \IR)$, they preserve constants, addition and scalar multiplication and are therefore also linear.
\end{proof}

We are left with the question what happens when we only consider polymorphisms of finite arity? 
To answer this question, let us point out that the main difference between finite arity and infinite arity polymorphisms is that a finite arity polymorphism preserve not only infinitary pp-definable sets but also directed unions of them. We make this precise in Theorem \ref{TheoremForPolymorphisms}. This gives the following picture:
$$
\begin{tikzpicture} 
    \node (left A) at (0,0)  [align=center,text width=4cm, draw] {infinitary \\ pp-definable} ;
    \node (left C) at (0,-4) [align=center,text width=4cm, draw] {directed union of \\ infinitary pp-definable};
    \node (right A) at (6,0)  [align=center,text width=4cm, draw] {preserved by  all \\ polymorphisms};
    \node (right B) at (6,-2) [align=center,text width=4cm, draw] {preserved by  countable \\ arity polymorphisms};
    \node (right C) at (6,-4) [align=center,text width=4cm, draw] {preserved by  finite \\ arity polymorphisms};
    \draw
        (left A) edge[-Latex] (right A)
        (left C) edge[-Latex] (right C)
        (left A) edge[-Latex] (left C)
        (right A) edge[-Latex] (right B)
        (right B) edge[-Latex] (right C)
        ;
\end{tikzpicture}
$$

\begin{thm} \label{TheoremForPolymorphisms}
    Let $\structA$ be a $\sigma$-structure and $S\subset \setstructA^n$. The following are equivalent:
    \begin{enumerate}
        \item \label{TheoremForPolymorphismsNR1}
        There is a directed set $N\subset \powerset(S)$ ordered by inclusion of  infinitary pp-definable sets (over $\structA$) such that $S = \bigcup_{S' \in N} S'$.
        \item \label{TheoremForPolymorphismsNR2}
        For each finite subset $F\subset S$, there is a set $S'$ between $F$ and $S$ which is infinitary pp-definable.
    \end{enumerate}
    Each set of this type has the following property:
    \begin{enumerate}[label=\Alph*., ref=\Alph*] 
        \item \label{TheoremForPolymorphismsNR3}
        The set is preserved under finite arity polymorphisms.
    \end{enumerate}
\end{thm}
\begin{proof} % of {Theorem \ref{TheoremForPolymorphisms}}
    \ref{TheoremForPolymorphismsNR1} implies \ref{TheoremForPolymorphismsNR2}: For each $f\in F$ choose any $S_f \in N$ such that $f\in S_f$. Choose $S'$ as a common successor of those elements. This exists since $N$ is a net.
    
    \ref{TheoremForPolymorphismsNR2} implies \ref{TheoremForPolymorphismsNR1}: For each finite subset $F$, let $S'_F$ be an infinitary pp-definable set containing $F$. Define $S_F\coloneqq \bigcap_{G\supset F, \text{ $G$ finite}} S'_G$ and note that this is an infinitary pp-definition. Define $N$ as the net $\{S_F\mid F\subset \setstructA \text{ finite}\}$. This has all properties: A successor of $S_F$ and $S_G$ is given by $S_{F\cup G}$ for $F,G$ finite subsets of $\setstructA$.
    
    \ref{TheoremForPolymorphismsNR2} implies \ref{TheoremForPolymorphismsNR3}: Let $p\colon \setstructA^m\to \setstructA$ be a finite arity polymorphism and take $a=(a_1,\dots ,a_m)\in S^m\subset (\setstructA^n)^m$. Then, there is $S'\subset S$ infinitary pp-definable such that $\{a_1,\dots ,a_m\} \subset S'$. Since $S'$ is polymorphism invariant by Lemma \ref{LemmaDefinablePolyStable} and $a\in (S')^m$, it follows that $p^n(a)\in S'\subset S$ where $p^n$ is considered as the map from $(\setstructA^n)^m$ to $\setstructA^n$ applying $p$ column wise. Thus, $S$ is invariant under polymorphisms.
\end{proof}
\begin{example}
    The implication \ref{TheoremForPolymorphismsNR3} to \ref{TheoremForPolymorphismsNR2} of Theorem \ref{TheoremForPolymorphisms} holds for countable structures \cite[Theorem 6.1.10]{theBodirsky}. We want to notice that proof there is not by the author of the book but by Marcello Mamino as explained in \cite[Section 6.1.2]{theBodirsky}. But the implication is not true in general:
    
    As an example, take the structure $(\setstructA,<)$ with the set $\setstructA=\{x \in \IR \mid x \in \IQ \lor x>0\}$ and the inherited relation $<$ from the real numbers. Every finite arity polymorphism will preserve the subset of the positive numbers as they have uncountably many lower numbers. 
    However, this set is not a directed union of infinitary pp-definable sets: It is easy to check that the nonempty infinitary pp-definable subsets of $\setstructA^n$ are given by identifying some of the $n$ variables and defining a partial order on the quotients. So $\setstructA$ itself is the only nonempty infinitary pp-definable subset of $A$.
\end{example}

\begin{lemma} \label{LemmaIPPandUnions}
    Consider a set of relations $\tau$ on $\IR$ which contains the structure $(\IR;+,\cdot c \mid c \in \IR, c\mid c \in \IR)$ and a non-affine convex set. Suppose furthermore that $\tau$ is closed under infinitary pp-definitions and contains all sets $S$ where there is a net $N\subset \powerset(S)\cap \tau$ ordered by inclusion such that $\bigcup_{S' \in N} S' = S$. 
    Then, $\tau$ contains all convex sets.
\end{lemma}
\begin{proof}
    Every compact interval $[0,n]$ for $n\in \IN$ is infinitary pp-definable from $\tau$ by Theorem \ref{TheoremDescriptionOfEqClasses} as $\tau$ contains a non-affine convex set. 
    Now, the closed ray $\{x\in \IR \mid x\ge 0\}$ is the directed union $\bigcup_{n\in \IN} [0,n]$, so it is also in $\tau$. Since every convex set is infinitary pp-definable from the ray by Theorem \ref{TheoremDescriptionOfEqClasses} or \ref{TheoremDefiningFromRay}, they are all in $\tau$.
\end{proof}

This lemma allows us to go the final step and prove the last remaining theorem from Section \ref{SectionResultsShort}:

\begin{thm}[Theorem \ref{TheoremClassificationUpPolymorphisms}] 
    Let $C$ be a locally closed clone on $\IR$ such that $C$ is contained in the finite arity polymorphisms of $(\IR;+,\cdot c \mid c\in \IR, c\mid c \in \IR)$ and $C$ contains the finite arity polymorphisms of the structure of all convex relations. Then $C$ is one of those two clones:
    \begin{enumerate}
        \item The polymorphism clone of $(\IR;+,\cdot c \mid c\in \IR, c\mid c \in \IR)$ respectively the polymorphism clone of all affine relations.
        In this clone, all elements $f\colon \IR^n \to \IR$ can be described as
        $$
            (a_1,a_2,\dots ,a_n) \mapsto (\lambda_1a_1+\lambda_2a_2+\dots +\lambda_n a_n)
        $$
        where $\lambda_1,\lambda_2,\dots ,\lambda_n\in \IR$ can be any constants such that $\lambda_1+\lambda_2+\dots +\lambda_n = 1$.
        \item The polymorphism clone of the structure containing all convex relations.
        In this clone, all elements $f\colon \IR^n \to \IR$ can be described as
        $$
            (a_1,a_2,\dots ,a_n) \mapsto (\lambda_1a_1+\lambda_2a_2+\dots +\lambda_n a_n)
        $$
        where $\lambda_1,\lambda_2,\dots ,\lambda_n\ge 0$ and $\lambda_1+\lambda_2+\dots +\lambda_n = 1$.
    \end{enumerate}
    In particular when $p$ is a finite arity polymorphism of $(\IR;+,\cdot c \mid c\in \IR, c\mid c \in \IR)$ that preserves any non-affine convex set, then it preserves all convex sets.
\end{thm}

\begin{proof}
    The set of all maps 
    $f\colon \IR^n \to \IR, (a_1,a_2,\dots ,a_n) \mapsto (\lambda_1a_1+\lambda_2a_2+\dots +\lambda_n a_n)$ where the sum of the lambda is 1
    is clearly contained in the set of all finite arity polymorphisms of affine sets which is contained in the polymorphism clone of $(\IR;+,\cdot c \mid c\in \IR, c\mid c \in \IR)$. 
    Conversely assume that $f\colon \IR^n \to \IR$ preserves $(\IR;+,\cdot c \mid c\in \IR, c\mid c \in \IR)$. 
    Then, it needs to be a vector space homomorphism and is thus given by $(a_1,a_2,\dots ,a_n) \mapsto \lambda_1a_1+\lambda_2a_2+\dots +\lambda_n a_n$. As it also preserves constants, the sum of the coefficients need to be equal to 1.

    Similar, the set of all maps $f\colon \IR^n \to \IR, (a_1,a_2,\dots ,a_n) \mapsto \lambda_1a_1+\lambda_2a_2+\dots +\lambda_n a_n$, where the sum of the coefficients lambda is 1 and every coefficient is nonnegative, is contained in the polymorphism clone of all convex sets by the definition of a convex set. Conversely, the polymorphism clone of all convex sets needs to be contained in the polymorphism clone of all affine sets and thus has the above description, where the sum of the coefficients is 1. To show that every coefficient is nonnegative, note that the polymorphism has to preserve the set $\{x\in \IR \mid x\ge 0\}$ of nonnegative numbers. Thus, for all points $(a_1,\dots ,a_n)$ with nonnegative coordinates, also the sum $\lambda_1a_1+\lambda_2a_2+\dots +\lambda_n a_n$ needs to be nonnegative. That is only possible if all coefficients $\lambda_i$ are nonnegative as we see by taking $a_i=1$ and all other coordinates to be zero.

    Clearly, both of these clones are locally closed clones by Theorem \ref{TheoremLocallyClosedClone}.

    Conversely a locally closed clone between them would be the polymorphism clone of any family $\tau$ of subsets of $\IR^n$. 
    We may assume that $\tau$ is the set of all relations which are invariant under the clone. Then, $\tau$ contains only convex sets, and contains all relations from $(\IR;+,\cdot c \mid c\in \IR, c\mid c \in \IR)$. Furthermore, $\tau$ is closed under infinitary pp-definability and directed unions by Theorem \ref{TheoremForPolymorphisms}. 
    If $\tau$ contains only affine sets, then the clone is the polymorphism clone of $(\IR;+,\cdot c \mid c\in \IR, c\mid c \in \IR)$ respectively the polymorphism clone of all affine relations as shown in the first paragraph of this proof. Otherwise, $\tau$ contains any non-affine convex set. In that case, we get by Lemma \ref{LemmaIPPandUnions} that $\tau$ contains all convex sets. Thus, the clone is the polymorphism clone of the set of all convex sets which is the second option.
\end{proof}

%% file: 07Thanks.tex
%Danke an Jan, Manuel, Pococop?, Andrew, Florian, Marga, EU

\section*{Acknowledgements}
I thank Prof. Manuel Bodirsky for his constant, polite and productive supervision and help. I thank Jan Thom who did the amazing job in reading the whole text in very a short time. Moreover, thank you Andrew Moorhead and Florian Starke for suggesting minor improvements and thank you Margarete Ketelsen. You asked me what this paper is about. Formulating out an answer was very helpful to write an introduction.

Lastly, I am really grateful for the European Union and the state of Saxony for paying me money.

%% file: main.bbl
\begin{thebibliography}{Hod93}

\bibitem[Ban55]{BanachReprint}
Stefan Banach.
\newblock {\em Th\'{e}orie des op\'{e}rations lin\'{e}aires}.
\newblock Chelsea Publishing Co., New York, 1955.
\newblock Reprint from 1932.

\bibitem[Bod21]{theBodirsky}
Manuel Bodirsky.
\newblock {\em Complexity of infinite-domain constraint satisfaction},
  volume~52 of {\em Lecture Notes in Logic}.
\newblock Cambridge University Press, Cambridge; Association for Symbolic
  Logic, Ithaca, NY, 2021.

\bibitem[Bor95]{BorelHeine}
Emile Borel.
\newblock Sur quelques points de la th\'{e}orie des fonctions.
\newblock {\em Ann. Sci. \'{E}cole Norm. Sup. (3)}, 12:9--55, 1895.

\bibitem[Car07]{MR1511425}
C.~Carath{\' e}odory.
\newblock \"{U}ber den {V}ariabilit\"{a}tsbereich der {K}oeffizienten von
  {P}otenzreihen, die gegebene {W}erte nicht annehmen.
\newblock {\em Math. Ann.}, 64(1):95--115, 1907.

\bibitem[Gei68]{Geiger}
David Geiger.
\newblock Closed systems of functions and predicates.
\newblock {\em Pacific J. Math.}, 27:95--100, 1968.

\bibitem[Hah22]{Hahn}
Hans Hahn.
\newblock \"{U}ber {F}olgen linearer {O}perationen.
\newblock {\em Monatsh. Math. Phys.}, 32(1):3--88, 1922.

\bibitem[Hod93]{Hodges}
Wilfrid Hodges.
\newblock {\em Model theory}, volume~42 of {\em Encyclopedia of Mathematics and
  its Applications}.
\newblock Cambridge University Press, Cambridge, 1993.

\bibitem[P{\" o}s04]{PoePotsd01}
R.~P{\" o}schel.
\newblock Galois connections for operations and relations.
\newblock In {\em Galois connections and applications}, volume 565 of {\em
  Math. Appl.}, pages 231--258. Kluwer Acad. Publ., Dordrecht, 2004.

\bibitem[Sze86]{Szendrei}
\'{A}gnes Szendrei.
\newblock {\em Clones in universal algebra}, volume~99 of {\em S\'{e}minaire de
  Math\'{e}matiques Sup\'{e}rieures [Seminar on Higher Mathematics]}.
\newblock Presses de l'Universit\'{e} de Montr\'{e}al, Montreal, QC, 1986.

\bibitem[Tyc30]{Tychonoff}
A.~Tychonoff.
\newblock \"{U}ber die topologische {E}rweiterung von {R}\"{a}umen.
\newblock {\em Math. Ann.}, 102(1):544--561, 1930.

\bibitem[Yos95]{MR1336382}
K{\= o}saku Yosida.
\newblock {\em Functional analysis}.
\newblock Classics in Mathematics. Springer-Verlag, Berlin, 1995.
\newblock Reprint of the sixth (1980) edition.

\bibitem[Z{\u a}l02]{MR1921556}
C.~Z{\u a}linescu.
\newblock {\em Convex analysis in general vector spaces}.
\newblock World Scientific Publishing Co., Inc., River Edge, NJ, 2002.

\end{thebibliography}
